\documentclass[a4paper,12pt]{amsart}
\usepackage{amsmath, amssymb}
\usepackage[abbrev]{amsrefs}
\usepackage{color}
\usepackage{hyperref}
\hypersetup{colorlinks=true,linkcolor=blue,citecolor=blue}

\newtheorem{thm}{Theorem}[section]
\newtheorem{prop}[thm]{Proposition}
\newtheorem{lem}[thm]{Lemma}
\newtheorem{cor}[thm]{Corollary}
\newtheorem{clm}[thm]{Claim}

\theoremstyle{definition}
\newtheorem{dfn}[thm]{Definition}
\newtheorem{ex}[thm]{Example}

\theoremstyle{remark}
\newtheorem{rem}[thm]{Remark}

\numberwithin{equation}{section}

\newcommand{\N}{\ensuremath{\mathbb{N}}}

\newcommand{\R}{\ensuremath{\mathbb{R}}}
\newcommand{\C}{\ensuremath{\mathbb{C}}}
\newcommand{\Hb}{\ensuremath{\mathbb{H}}}
\newcommand{\X}{\mathcal{X}}

\newcommand{\cP}{\mathcal{P}}

\newcommand{\cB}{\mathcal{B}}
\newcommand{\cD}{\mathcal{D}}
\newcommand{\tC}{\widetilde{C}}
\newcommand{\tK}{\widetilde{K}}
\newcommand{\tX}{\widetilde{X}}
\newcommand{\tY}{\widetilde{Y}}
\newcommand{\cX}{\mathcal{X}}

\newcommand{\E}{\mathbb{E}}
\newcommand{\cI}{\mathcal{I}}

\DeclareMathOperator{\supp}{supp}
\DeclareMathOperator{\diam}{diam}
\DeclareMathOperator{\vol}{vol}
\newcommand{\midd}{\mathrel{} \middle| \mathrel{}}

\newcommand{\kf}{d_{\mathrm{KF}}}
\newcommand{\prok}{d_{\mathrm{P}}}
\newcommand{\Lip}{{\mathcal{L}}ip}
\newcommand{\conc}{d_{\mathrm{conc}}}
\newcommand{\haus}{d_{\mathrm{H}}}
\newcommand{\id}{\mathrm{id}}
\newcommand{\Id}{\mathrm{Id}}

\newcommand{\lm}{\mathrm{lm}}
\newcommand{\proj}{\mathrm{proj}}

\newcommand{\sk}{\boldsymbol{s}_\kappa}
\newcommand{\ck}{\boldsymbol{c}_\kappa}
\newcommand{\tsk}{\tilde{\boldsymbol{s}}_\kappa}

\DeclareMathOperator{\CD}{CD}
\DeclareMathOperator{\OD}{ObsDiam}
\DeclareMathOperator{\PD}{diam}
\DeclareMathOperator{\Ric}{Ric}
\DeclareMathOperator{\Hess}{Hess}
\DeclareMathOperator{\lip}{lip}

\newcommand{\ep}{\varepsilon}
\renewcommand{\phi}{\varphi}

\makeatletter
\@namedef{subjclassname@2020}{
\textup{2020} Mathematics Subject Classification}
\makeatother

\title[Convergence of cones of metric measure spaces]{Convergence of cones of metric measure spaces \\ and its application to Cauchy distribution}

\author[S.~Esaki]{Syota Esaki}
\address{Department of Applied Mathematics, Fukuoka University, Fukuoka 814-0180, JAPAN}
\email{sesaki@fukuoka-u.ac.jp}

\author[D.~Kazukawa]{Daisuke Kazukawa}
\address{Faculty of Mathematics, Kyushu University, Fukuoka 819-0395, JAPAN}
\email{kazukawa@math.kyushu-u.ac.jp}

\author[A.~Mitsuishi]{Ayato Mitsuishi}
\address{Department of Applied Mathematics, Fukuoka University, Fukuoka 814-0180, JAPAN}
\email{mitsuishi@fukuoka-u.ac.jp}

\date{February 22, 2024}

\begin{document}

\begin{abstract}
We prove that the sequence of cones of metric measure spaces converges if the sequence of base spaces converges in Gromov's box, concentration, and weak topologies.
As an application, we show that the generalized Cauchy distribution with suitable scaling converges to a half line in the concentration topology as the dimension diverges to infinity.
This is a new example distinguished from previously known examples such as Gaussian distributions and typical closed Riemannian manifolds with constant Ricci curvature.
\end{abstract}

\maketitle

\section{Introduction}

The concentration of measure phenomenon has been attracting attention in recent years.
It was first discovered by P.~L\'evy for the unit spheres, and then was put forward by V.~Milman in the 1970's.
\cites{Levy, VMil, Led} are listed for reference.
Roughly speaking, the concentration of measure phenomenon is explained as ``any $1$-Lipschitz function is almost a constant", which is often observed for high-dimensional spaces.
For example, on the $n$-dimensional unit sphere $S^n$ in $\R^{n+1}$, we have
\begin{equation}\label{eq:intro_conc}
\sigma^n(\left\{\theta \in S^n \midd |f(\theta)-m_f| \ge \ep \right\}) \le 2e^{-\frac{n-1}{2}\ep^2}
\end{equation}
for any $1$-Lipschitz function $f$ with median $m_f$ and any $\ep>0$, where $\sigma^n$ is the uniform probability measure on $S^n$.
This phenomenon strongly reflects the geometry of high dimensions.
Taking a $1$-Lipschitz function $f$ on $S^n$ as the distance from a Borel subset $A$ with $\sigma^n(A) \ge \frac{1}{2}$, \eqref{eq:intro_conc} implies
\[
\sigma^n(U_\ep(A)) \ge 1-2e^{-\frac{n-1}{2}\ep^2},
\]
where $U_\ep(A)$ is the open $\ep$-neighborhood of $A$.
More generally, this phenomenon is observed for closed Riemannian manifolds whose Ricci curvature diverges to infinity (see \cite{GM}).

For the study of the concentration of measure phenomenon, Gromov provided a revolutionary concept in \cite{Grmv}*{Chapter $3.\frac{1}{2}_+$}.
He introduced a distance function, called the {\it observable distance}, on the set $\X$ of isomorphism classes of metric measure spaces.
The induced topology is called the {\it concentration topology}.
The classical concentration of measure phenomenon corresponds to convergence to a one-point metric measure space with respect to the concentration topology.
Such a sequence of metric measure spaces is called a {\it L\'evy family}.
For example, the sequence $\{S^n\}_{n=1}^\infty$ of unit spheres is a L\'evy family.
Thus the concentration topology gives a characteristic notion of convergence of metric measure spaces capturing the high-dimensional aspects. It is also interesting as a generalization of the celebrated measured Gromov-Hausdorff convergence.
Indeed, any measured Gromov-Hausdorff convergent sequence concentrates (i.e., converges with respect to the concentration topology).

Furthermore, Gromov also introduced a natural compactification, denoted by $\Pi$, of $\X$ with respect to the concentration topology at the same time, which is one of powerful tools to study the concentration topology.
The topology of this compactification is called the {\it weak topology} and each element of $\Pi$ is called a {\it pyramid}.
The space $\Pi$ contains many infinite-dimensional objects, for example, the (virtual) infinite-dimensional Gaussian space and infinite-dimensional cube.
Pyramids themselves are also worth studying.

In recent years, the study of the concentration and weak topologies is rapidly growing.
Especially, the discovery of new concentrating sequences is significant as the generalization of the concentration of measure phenomenon.
Shioya \cites{MML,MMG} proved that the $n$-dimensional sphere $S^n(\sqrt{n})$ in $\R^{n+1}$ with radius $\sqrt{n}$ converges to the infinite-dimensional Gaussian space in the compactification $\Pi$.
In addition, the limit spaces of such as projective spaces, ellipsoids, and Stiefel manifolds whose dimension diverges to infinity have been determined under appropriate scales (see \cites{comts, ellipsoid, ST}).

In this paper, we study the convergence of cones of metric measure spaces and pyramids and apply the result to find new examples of concentrating sequences.
We start by describing our setting and main results.

\begin{dfn}[$\kappa$-Cone]\label{dfn:cone}
The {\it $\kappa$-cone} $K_\kappa(X)$ of a metric space $(X, d_X)$ for $\kappa \in \R$ is the quotient space \[
K_\kappa(X) := I_\kappa \times X \,/ \, \partial I_\kappa\times X
\]
with metric $d_{K_\kappa(X)}$, where $I_\kappa$ is the interval given by
\[
I_{\kappa} := \begin{cases}
[0, \pi/\sqrt{\kappa}] & \text{if } \kappa>0, \\
[0,\infty) & \text{if } \kappa \le 0,
\end{cases}
\]
and $\partial I_\kappa$ is the set of end points of $I_\kappa$, that is, $\partial I_\kappa = \{0, \pi/\sqrt{\kappa}\}$ if $\kappa > 0$ and $\partial I_\kappa = \{0\}$ if $\kappa \le 0$.
Let $q \colon I_\kappa \times X \to K_\kappa(X)$ be the quotient map.
The metric $d_{K_\kappa(X)}$ is given by
\begin{equation}\label{eq:cone_dist}
\ck(d_{K_\kappa(X)}(rx, r'x')) = \ck(r)\ck(r')+ \kappa\sk(r)\sk(r')\cos(\min\{d_X(x,x'), \pi\})
\end{equation}
for $rx=q(r,x), r'x'=q(r',x') \in K_\kappa(X)$, where $\sk$ is the unique solution to
\[
s''+\kappa s= 0, \quad  s(0) = 0, \quad s'(0) = 1,
\]
and $\ck$ is defined by $\ck := \sk'$.
Note that $\sk, \ck$ are explicitly expressed as
\[
\sk(r) :=
\begin{cases}
\frac{1}{\sqrt{\kappa}}\sin(\sqrt{\kappa}r) & \text{if } \kappa>0,\\
r & \text{if } \kappa=0,\\
\frac{1}{\sqrt{-\kappa}}\sinh(\sqrt{-\kappa}r) & \text{if } \kappa<0,
\end{cases} \quad
\ck(r) :=
\begin{cases}
\cos(\sqrt{\kappa}r) & \text{if } \kappa>0,\\
1 & \text{if } \kappa=0,\\
\cosh(\sqrt{-\kappa}r) & \text{if } \kappa<0.
\end{cases}
\]
In the case of $\kappa = 0$, we interpret \eqref{eq:cone_dist} as the limit formula
\[
d_{K_0(X)}(rx, r'x')^2 = r^2+(r')^2-2rr'\cos(\min\{d_X(x, x'), \pi\})
\]
as $\kappa\to 0$.
Moreover, for a metric measure space $(X, d_X, m_X)$ and a Borel probability measure $\mu$ on $I_\kappa$,
we define a metric measure space $K_{\kappa,\mu}(X)$ by the metric space $(K_\kappa(X), d_{K_\kappa(X)})$ equipped with the push-forward measure $q_\#(\mu\otimes m_X)$ and call it the {\it $\kappa$-cone of $X$ with respect to $\mu$}.
\end{dfn}

Furthermore, for any pyramid $\cP \in \Pi$, we can define the $\kappa$-cone $K_{\kappa, \mu}(\cP) \in \Pi$ of $\cP$ naturally.
See Definition \ref{dfn:cone_py} for the precise definition of $K_{\kappa, \mu}(\cP)$.

The following theorem is the main result in this paper.

\begin{thm}\label{thm:cone_py}
Let $\kappa \in \R$.
Assume that a sequence $\{\cP_n\}_{n=1}^\infty$ of pyramids converges weakly to a pyramid $\cP$ and that a sequence $\{\mu_n\}_{n=1}^\infty$ of Borel probability measures on $I_\kappa$ converges weakly to a Borel probability measure $\mu$.
Then the sequence $\{K_{\kappa, \mu_n}(\cP_n)\}_{n=1}^\infty$ of $\kappa$-cones converges weakly to $K_{\kappa, \mu}(\cP)$.
\end{thm}

In particular, we obtain the following as a corollary of Theorem \ref{thm:cone_py} directly.

\begin{thm}\label{thm:main}
Let $\kappa \in \R$ and let $\{X_n\}_{n=1}^\infty$ be a L\'evy family.
Assume that a sequence $\{\mu_n\}_{n=1}^\infty$ of Borel probability measures on $I_\kappa$ converges weakly to a Borel probability measure $\mu$.
Then the sequence $\{K_{\kappa, \mu_n}(X_n)\}_{n=1}^\infty$ of $\kappa$-cones concentrates to the interval $(I_\kappa, |\cdot|, \mu)$.
\end{thm}

Note that the interval $(I_\kappa, |\cdot|, \mu)$ is isomorphic to the $\kappa$-cone of a one-point metric measure space with respect to $\mu$.
In Section \ref{sec:main_proof}, we also give another proof to Theorem \ref{thm:main} in a very simple way that is limited to this case.
In this proof, we prove that ``any $1$-Lipschitz function on $K_{\kappa,\mu_n}(X_n)$ is almost a radial function if $n$ is large".
This is exactly a philosophy based on the concentration of measure phenomenon.

As an application of our results, we discover a very natural new example of concentrating sequences.
This is the generalized Cauchy distribution with one-parameter for concavity.
For a given $\beta > 0$, we define a Borel probability measure $\nu^n_\beta$ on the Euclidean space $(\R^n, \|\cdot\|)$ by
\begin{equation}\label{eq:Cauchy}
\nu^n_\beta(dx) := \frac{\Gamma(\frac{n+\beta}{2})}{\pi^{\frac{n}{2}}\Gamma(\frac{\beta}{2})}\frac{dx}{(1+\|x\|^2)^\frac{n+\beta}{2}}
\end{equation}
and call it the {\it $n$-dimensional generalized Cauchy distribution} (or {\it measure}) {\it with exponent $\beta$}.
This name is quoted from \cite{BL} and this is also called the multivariate $t$-distribution in some contexts.
From now on, we call them the Cauchy distributions simply unless otherwise stated.
We denote by $C^n_\beta$ the $n$-dimensional Euclidean space equipped with the Cauchy distribution $\nu^n_\beta$ and call $C^n_\beta$ {\it the $n$-dimensional Cauchy space with exponent $\beta$}.

The Cauchy distribution with $\beta=1$ is well-known along with the Gaussian distribution.
The $n$-dimensional Cauchy distribution $\nu^n_1$ is the law of random vector $(X_1/X_0,\ldots,X_n/X_0)$, where random variables $X_i$, $i=0,1,\ldots,n$, are independent and identically distributed by the standard Gaussian distribution.
The distribution $\nu^n_1$ is one of the stable distributions, as is the Gaussian distribution.
Moreover, it is also used in statistics and physics, so is of particular important in recent years.

We obtain the following result.
Here, we denote by $tX$ the $t$-scaled space $(X, td_X, m_X)$ of a metric measure space $X = (X, d_X, m_X)$ for $t > 0$.

\begin{thm}\label{thm:main_cauchy}
The sequence $\{\frac{1}{\sqrt{n}} C^n_\beta\}_{n=1}^\infty$ concentrates to the half line $([0, \infty), |\cdot|)$ equipped with the Borel probability measure
\begin{equation}\label{eq:recip}
\nu_\beta(dt) := \frac{2^{1-\frac{\beta}{2}}}{\Gamma(\frac{\beta}{2})} \frac{1}{t^{\beta+1}} e^{-\frac{1}{2t^2}} \, dt.
\end{equation}
\end{thm}

Surprisingly, the sequence $\{\frac{1}{\sqrt{n}} C^n_\beta\}_{n=1}^\infty$ concentrates to a one-dimensional space, neither to a one-point nor to an infinite-dimensional space.
Behind Theorem \ref{thm:main_cauchy}, there is the philosophy that ``any $(1/\sqrt{n})$-Lipschitz function on $C^n_\beta$ is almost a radial function if $n$ is large".
The limit measure $\nu_\beta$ is also of interest.
If $\beta=1$, the limit measure $\nu_1$ is called the absolutely reciprocal normal distribution, that is, $\nu_1$ is the law of the random variable $1/|X|$ provided that the law of $X$ is the standard one-dimensional Gaussian distribution.

Theorem \ref{thm:main_cauchy} can be stated more precisely as follows.

\begin{thm}\label{thm:main_cauchy_gene}
Let $\{r_n\}_{n=1}^\infty$ be a sequence of positive real numbers. Then we have the following {\rm (1) -- (3)}.
\begin{enumerate}
\item If $r_n\sqrt{n} \to 0$, then the sequence $\{r_n C^n_\beta\}_{n=1}^\infty$ box-converges to a one-point metric measure space.
\item  If $r_n\sqrt{n} \to \lambda \in (0,\infty)$, then the sequence $\{r_n C^n_\beta\}_{n=1}^\infty$ concentrates to the half line $([0, \infty), |\cdot|)$ equipped with the Borel probability measure
\[
\nu_{\beta, \lambda}(dt) := \frac{2^{1-\frac{\beta}{2}}}{\Gamma(\frac{\beta}{2})} \frac{\lambda^\beta}{t^{\beta+1}} e^{-\frac{\lambda^2}{2t^2}} \, dt.
\]
\item If $r_n\sqrt{n} \to \infty$, then the sequence $\{r_n C^n_\beta\}_{n=1}^\infty$ infinitely dissipates.
\end{enumerate}
\end{thm}

This theorem means that the $n$-dimensional Cauchy space has the phase transition property with the critical scale order $1/\sqrt{n}$.
See \cite{OS} for the definition of the phase transition property.
The box-convergence induced by the box distance function on $\X$ is a slight modification of the measured Gromov-Hausdorff convergence.
The box distance was also introduced by Gromov in \cite{Grmv}.
Any measured Gromov-Hausdorff convergent sequence box-converges and any box-convergent sequence concentrates.
The infinite dissipation is the opposite notion of the L\'evy family.
For details on this notion, see Subsection \ref{subsec:dissip} or refer to \cites{Grmv, MMG}.

As a corollary, we obtain a statement similar to the normal law \`a la L\'evy.

\begin{cor}\label{cor:normal}
Let $f_n\colon \R^n \to \R$, $n=1,2,\ldots$, be $(1/\sqrt{n})$-Lipschitz functions. Assume that, for a subsequence $\{f_{n_i}\}$ of $\{f_n\}$, the push-forward measure $(f_{n_i})_\# \nu^{n_i}_\beta$ converges weakly to a Borel probability measure $\mu$ on $\R$. Then, there exists a $1$-Lipschitz function $\alpha\colon[0,\infty)\to\R$ such that
\[
\alpha_\# \nu_\beta = \mu.
\]
\end{cor}

The usual normal law requires that $\alpha$ is a monotone function.
Unfortunately, our result does not lead to the monotonicity of $\alpha$.
This law is deeply related to the isoperimetric problem (see \cite{GroWaist}).

Moreover, we obtain the asymptotic behavior of the observable diameter, which is one of the most fundamental invariants for the concentration of metric measure spaces. The definition of the observable diameter is in Definition \ref{dfn:OD}.

\begin{cor}\label{cor:OD}
We have
\[
\lim_{n\to\infty} \OD(\frac{1}{\sqrt{n}} C^n_\beta;-\kappa) = \OD(([0,\infty), |\cdot|, \nu_\beta); -\kappa)
\]
for any $\kappa > 0$.
\end{cor}

On Theorem \ref{thm:main_cauchy}, it is also significant to know whether the sequence $\{\frac{1}{\sqrt{n}} C^n_\beta\}_{n=1}^\infty$ converges with respect to a topology stronger  than the concentration topology.
We also obtain the following theorem.

\begin{thm}\label{thm:non-box}
The sequence $\{\frac{1}{\sqrt{n}} C^n_\beta\}_{n=1}^\infty$ has no box-convergent {\rm (}hence no measured Gromov-Hausdorff convergent{\rm )} subsequence.
\end{thm}

As the final part of the introduction, we discuss the relation between Theorem \ref{thm:main_cauchy} and the weighted Ricci curvature.
For the precise definition, see Section \ref{sec:inv}.
We see that the $(-\beta)$-weighted Ricci curvature $\Ric_{\frac{1}{\sqrt{n}}C^n_\beta, -\beta}$ of the scaled Cauchy space $\frac{1}{\sqrt{n}}C^n_\beta$ is nonnegative, or equivalently, $\frac{1}{\sqrt{n}}C^n_\beta$ is a $\CD(0,-\beta)$ space.
The study of convergence of $\CD(K,N)$ spaces for $N<0$ is performed by Magnabosco-Rigoni-Sosa \cite{MRS} and Oshima \cite{Oshima}.
Theorem \ref{thm:main_cauchy} provides a natural example of the main result in \cite{Oshima}.
Moreover, we observe the asymptotic behavior of weighted Ricci curvature as $n\to\infty$.
At any point $x \in \R^n\setminus\{o\}$, we have
\[
\Ric_{\frac{1}{\sqrt{n}}C^n_\beta, -\beta}(\frac{x}{\|x\|},\frac{x}{\|x\|}) = \frac{n(n+\beta)}{(1+n\|x\|^2)^2} \to \frac{1}{\|x\|^4}
\]
and, for any $w \in \R^n\setminus\{o\}$ with $w \perp x$,
\[
\Ric_{\frac{1}{\sqrt{n}}C^n_\beta, -\beta}(\frac{w}{\|w\|},\frac{w}{\|w\|}) = \frac{n(n+\beta)}{1+n\|x\|^2} \to \infty
\]
as $n\to\infty$.
As mentioned at the beginning, the divergence of the Ricci curvature to infinity implies the L\'evy family (see \cite{Oshima}*{Corollary 4.3} for $N<0$).
Therefore Theorem \ref{thm:main_cauchy} can be interpreted as that the $(n-1)$-dimensional spherical directions at each point of $\frac{1}{\sqrt{n}}C^n_\beta$ concentrates to a single point leaving only the $1$-dimensional radial direction.

The organization of this paper is as follows.
After the preliminaries section,
in Section \ref{sec:cone_conv}, we prove Theorem \ref{thm:cone_py}.
The proof there is close to the discussion in \cites{prod, comts}.
Moreover, in Section \ref{sec:main_proof}, we give another proof to Theorem \ref{thm:main}, which restricts Theorem \ref{thm:cone_py} to L\'evy families.
In the proof of Theorem \ref{thm:main}, we exactly prove that ``any $1$-Lipschitz function on a cone is almost a radial function if $n$ is large" in \eqref{eq:main_proof} based on the concentration of measure phenomenon.
From Section \ref{sec:cauchy} to \ref{sec:ergodic}, we mainly consider the Cauchy distribution.
At first, in Section \ref{sec:cauchy}, we describe fundamental properties of Cauchy distribution and prove Theorem \ref{thm:main_cauchy} applying Theorem \ref{thm:main}.
The refinement from Theorem \ref{thm:main_cauchy} to Theorem \ref{thm:main_cauchy_gene} is a simple argument based on previous works.
In Section \ref{sec:non-box}, we prove Theorem \ref{thm:non-box}.
The key is the continuity of the partial diameter with respect to the box distance.
After the proof of our results, in Section \ref{sec:inv}, we calculate several geometric and analytic quantities for the Cauchy space and the limit space.
Finally, in Section \ref{sec:ergodic}, we remark a relation to the ergodic theory.
We give some examples of non-ergodic systems generated by the Cauchy distribution.

\subsection*{Acknowledgement}
The authors would like to thank Professors Shouhei Honda, Atsushi Katsuda, Kazuhiro Kuwae, Hiroki Nakajima, Takashi Shioya, and Asuka Takatsu for their valuable comments and encouragement.
S.E.~is supported in part by JSPS KAKENHI Grant Numbers JP22H04942 and JP23K03158.
D.K.~is supported in part by JSPS KAKENHI Grant Numbers JP22K20338.
A.M.~is supported in part by JSPS KAKENHI Grant Numbers JP20K03598 and JP21H00977.

\section{Preliminaries}

In this section, we describe the definitions and some properties of metric measure space, the box distance, and the observable distance.
We use most of these notions along \cite{MMG}.
As for more details, we refer to \cite{MMG} and \cite{Grmv}*{Chapter 3$\frac{1}{2}_+$}.

\subsection{Metric measure spaces}
Let $(X, d_X)$ be a complete separable metric space and $m_X$ a Borel probability measure on $X$.
We call the triple $(X, d_X, m_X)$ a {\it metric measure space}, or an {\it mm-space} for short.
We sometimes say that $X$ is an mm-space, in which case the metric and the measure of $X$ are respectively indicated by $d_X$ and $m_X$.

\begin{dfn}[mm-Isomorphism]
Two mm-spaces $X$ and $Y$ are said to be {\it mm-isomorphic} to each other if there exists an isometry $f \colon \supp{m_X} \to \supp{m_Y}$ such that $f_\# m_X = m_Y$, where $f_\# m_X$ is the push-forward measure of $m_X$ by $f$. Such an isometry $f$ is called an {\it mm-isomorphism}. Denote by $\mathcal{X}$ the set of mm-isomorphism classes of mm-spaces.
\end{dfn}

Note that an mm-space $X$ is mm-isomorphic to $(\supp{m_X}, d_X , m_X)$. We assume that an mm-space $X$ satisfies
\[
X = \supp{m_X}
\]
unless otherwise stated.

\begin{dfn}[Lipschitz order]
Let $X$ and $Y$ be two mm-spaces. We say that $X$ ({\it Lipschitz}) {\it dominates} $Y$ and write $Y \prec X$ if there exists a $1$-Lipschitz map $f \colon X \to Y$ satisfying $f_\# m_X = m_Y$. We call the relation $\prec$ on $\X$ the {\it Lipschitz order}.
\end{dfn}

The Lipschitz order $\prec$ is a partial order relation on $\X$.

\subsection{Box distance and observable distance}
For a subset $A$ of a metric space $(X, d_X)$ and for a real number $r > 0$, we set
\[
U_r(A) := \{x \in X \mid d_X(x, A) < r\},
\]
where $d_X(x, A) := \inf_{a \in A} d_X(x, a)$.

\begin{dfn}[Prokhorov distance]
The {\it Prokhorov distance} $\prok(\mu, \nu)$ between two Borel probability measures $\mu$ and $\nu$ on a metric space $X$ is defined to be the infimum of $\varepsilon > 0$ satisfying
\[
\mu(U_\varepsilon(A)) \geq \nu(A) - \varepsilon
\]
for any Borel subset $A \subset X$.
\end{dfn}

The Prokhorov metric $\prok$ is a metrization of the weak convergence of Borel probability measures on $X$ provided that $X$ is a separable metric space. Note that if a map $f \colon X \to Y$ between two metric spaces $X$ and $Y$ is $1$-Lipschitz, then we have
\begin{equation}\label{eq:lip_prok}
\prok(f_\#\mu, f_\# \nu) \leq \prok(\mu, \nu)
\end{equation}
for any two Borel probability measures $\mu$ and $\nu$ on $X$.

\begin{dfn}[Ky Fan metric]
Let $(X, \mu)$ be a measure space and $(Y, d_Y)$ a metric space. For two $\mu$-measurable maps $f,g \colon X \to Y$, we define $\kf^\mu (f, g)$ to be the infimum of $\varepsilon \geq 0$ satisfying
\[
\mu(\{x \in X \mid d_Y(f(x),g(x)) > \varepsilon \}) \leq \varepsilon.
\]
The function $\kf^\mu$ is a metric on the set of $\mu$-measurable maps from $X$ to $Y$ by identifying two maps if they are equal to each other $\mu$-almost everywhere. We call $\kf^\mu$ the {\it Ky Fan metric}.
\end{dfn}

\begin{lem}[\cite{MMG}*{Lemma 1.26}]\label{lem:prok_kf}
Let $X$ be a topological space with a Borel probability measure $\mu$ and $Y$ a metric space. For any two Borel measurable maps $f, g \colon X \to Y$, we have
\[
\prok(f_\#\mu, g_\# \mu) \leq \kf^\mu (f, g).
\]
\end{lem}

\begin{dfn}[Parameter]
Let $X$ be an mm-space. A map $\varphi \colon [0,1] \to X$ is called a {\it parameter} of $X$ if $\varphi$ is a Borel measurable map such that
\begin{equation*}
\varphi_\# \mathcal{L}^1 = m_X,
\end{equation*}
where $\mathcal{L}^1$ is the one-dimensional Lebesgue measure on $[0,1]$.
\end{dfn}

Note that any mm-space has a parameter (see \cite{MMG}*{Lemma 4.2}).

\begin{dfn}[Box distance]
We define the {\it box distance} $\square(X, Y)$ between two mm-spaces $X$ and $Y$ to be the infimum of $\varepsilon \geq 0$ satisfying that there exist parameters $\varphi \colon [0,1]  \to X$, $\psi \colon [0,1]  \to Y$, and a Borel subset $J \subset [0,1]$ with $\mathcal{L}^1(J) \geq 1 - \varepsilon$ such that
\[
|d_X(\varphi(s), \varphi(t)) - d_Y(\psi(s), \psi(t))| \le \ep
\]
for any $s,t \in J$.
\end{dfn}

\begin{thm}[\cite{MMG}*{Theorem 4.10}]
The box distance function $\square$ is a complete separable metric on $\mathcal{X}$.
\end{thm}

\begin{lem}[\cite{MMG}*{Proposition 4.12}]\label{prop:mmg4.12}
Let $X$ be a complete separable metric space. For any two Borel probability measures $\mu$ and $\nu$ on $X$, we have
\[
\square((X, \mu), (X, \nu)) \leq 2\prok(\mu, \nu).
\]
\end{lem}

The following notion gives one of the conditions that are equivalent to the box-convergence.

\begin{dfn}[$\ep$-mm-Isomorphism]
Let $X$ and $Y$ be two mm-spaces and $f \colon X \to Y$ a Borel measurable map.
We say that $f$ is an {\it $\ep$-mm-isomorphism} for a real number $\ep \ge 0$ if there exists a Borel subset $\tX \subset X$ such that
\begin{enumerate}
\item $m_X(\tX) \ge 1 - \ep$,
\item $|d_X(x, x') - d_Y(f(x),f(x'))| \le \varepsilon$ for any $x, x' \in \tX$,
\item $\prok(f_\# m_X, m_Y) \le \ep$.
\end{enumerate}
We call $\tX$ a {\it non-exceptional domain} of $f$.
\end{dfn}

\begin{lem}[\cite{MMG}*{Lemma 4.22}]\label{lem:ep-mm-isom}
Let $X$ and $Y$ be two mm-spaces.
\begin{enumerate}
\item If there exists an $\ep$-mm-isomorphism $f \colon X \to Y$, then $\square(X, Y) \le 3\ep$.
\item If $\square(X, Y) < \ep$, then there exists a $3\ep$-mm-isomorphism $f\colon X \to Y$.
\end{enumerate}
\end{lem}

Given a metric space $X$, we denote by $\Lip_1(X)$ the set of 1-Lipschitz functions on $X$.
Moreover, for a map $p \colon S \to X$ from a set $S$, we set
\begin{equation*}
p^* \Lip_1(X) := \left\{ f \circ p \midd f \in \Lip_1(X) \right\}.
\end{equation*}
Note that, for any parameter $\varphi \colon [0,1] \to X$ of an mm-space $X$, the set $\varphi^* \Lip_1(X)$ consists of Borel measurable functions on $[0,1]$.

\begin{dfn}[Observable distance]
We define the {\it observable distance} $\conc(X, Y)$ between two mm-spaces $X$ and $Y$ by
\begin{equation*}
\conc(X, Y) := \inf_{\varphi, \psi} \haus(\varphi^* \Lip_1(X), \psi^* \Lip_1(Y)),
\end{equation*}
where $\varphi \colon [0,1] \to X$ and $\psi \colon [0,1] \to Y$ run over all parameters of $X$ and $Y$ respectively, and $\haus$ is the Hausdorff distance with respect to the metric $\kf^{\mathcal{L}^1}$. We say that a sequence $\{X_n\}_{n=1}^\infty$ of mm-spaces  {\it concentrates} to an mm-space $X$ if $X_n$ $\conc$-converges to $X$ as $n \to \infty$.
\end{dfn}

\begin{thm}[\cite{MMG}*{Theorem 5.13, Proposition 5.5}]
The observable distance function $\conc$ is a metric on $\mathcal{X}$. Moreover, we have
\[
\conc(X, Y) \leq \square(X, Y)
\]
for any two mm-spaces $X$ and $Y$.
\end{thm}

\begin{lem}[\cite{MMG}*{Lemma 5.25}]\label{lem:enforce}
Let $p\colon X \to Y$ be a Borel measurable map. Then we have
\[
\conc(X, Y) \le 2\prok(p_\# m_X, m_Y) + \haus(\Lip_1(X), p^*\Lip_1(Y)),
\]
where $\haus$ is the Hausdorff distance with respect to the metric $\kf^{m_X}$.
\end{lem}

\subsection{L\'evy family}
The observable diameter is one of the most fundamental invariants of an mm-space.

\begin{dfn}[Partial and observable diameter]\label{dfn:OD}
Let $X$ be an mm-space. For a real number $\alpha \leq 1$, we define the {\it partial diameter} $\PD(X; \alpha)$ of $X$ to be the infimum of $\diam{A}$, where $A \subset X$ runs over all Borel subsets with $m_X(A) \geq \alpha$ and $\diam{A}$ is the diameter of $A$. For a real number $\kappa > 0$, we define the {\it observable diameter} of $X$ to be
\begin{equation}
\OD(X; -\kappa) := \sup_{f \in \Lip_1(X)} \PD((\R, |\cdot|, f_\# m_X); 1 - \kappa).
\end{equation}
\end{dfn}

The observable diameter is an invariant under mm-isomorphism. Note that $\diam(X; 1-\kappa)$ and $\OD(X; -\kappa)$ are nonincreasing in $\kappa > 0$.

\begin{dfn}[L\'evy family]
A sequence $\{X_n\}_{n=1}^\infty$ of mm-spaces is called a {\it L\'evy family} if
\begin{equation*}
\lim_{n \to \infty} \OD(X_n; -\kappa) = 0
\end{equation*}
for any $\kappa > 0$.
\end{dfn}

We denote by $*$ the one-point mm-space equipped with the trivial distance and the Dirac measure.

\begin{prop}[\cite{MMG}*{Corollary 5.8}]
A sequence $\{X_n\}_{n=1}^\infty$ of mm-spaces is a L\'evy family if and only if it concentrates to $*$ as $n \to \infty$.
\end{prop}

\begin{dfn}[Concentration function]
Let $X$ be an mm-space.
We define the {\it concentration function} $\alpha_X$ of $X$ to be
\begin{equation*}
\alpha_{X}(r) := \sup_A{(1 - m_X(U_r(A)))}
\end{equation*}
for $r > 0$, where $A \subset X$ runs over all Borel subsets with $m_X(A) \geq 1/2$.
\end{dfn}

\begin{prop}[\cite{Led}*{Proposition 1.12}, \cite{MMG}*{Remark 2.28}]\label{prop:conc_fct}
The following {\rm (1)} and {\rm (2)} hold.
\begin{enumerate}
\item $\OD(X; -\kappa) \leq 2\inf\{r > 0 \, | \, \alpha_X(r) \leq \kappa/2\}$ for any $\kappa > 0$.
\item $\alpha_X(r) \leq \sup\left\{\kappa >0  \midd \OD(X; -\kappa) \geq r\right\}$ for any $r > 0$.
\end{enumerate}
In particular, a sequence $\{X_n\}_{n=1}^\infty$ of mm-spaces is a L\'evy family if and only if
\[
\lim_{n\to\infty}\alpha_{X_n}(r) = 0
\]
for any $r>0$.
\end{prop}

\begin{ex}[\cite{Led}*{Theorem 2.3}]
Let $S^{n-1}$ be the $(n-1)$-dimensional unit sphere in $\R^n$. Assume that $S^{n-1}$ is equipped with the Riemannian distance or the restriction of Euclidean distance and the uniform probability measure $\sigma^{n-1}$. Then we have
\begin{equation}\label{eq:conc_sphere}
\alpha_{S^{n-1}}(r) \le e^{-\frac{n-2}{2}r^2}
\end{equation}
for any $r > 0$.
\end{ex}

\begin{dfn}[Median and L\'evy mean]
Let $X$ be a measure space with probability measure $\mu$ and $f \colon X \to \R$ a measurable function. A real number $m \in \R$ is called a {\it median} of $f$ if it satisfies
\begin{equation*}
\mu(\{ x \in X \mid  f(x) \geq m\}) \geq \frac{1}{2} \quad \text{and} \quad \mu(\{ x \in X \mid  f(x) \leq m\}) \geq \frac{1}{2}.
\end{equation*}
It is easy to see that the set of medians of $f$ is a nonempty bounded closed interval. The {\it L\'evy mean} $\lm(f; \mu)$ of $f$ with respect to $\mu$ is defined to be
\begin{equation*}
\lm(f; \mu) := \frac{\underline{m} +\overline{m}}{2},
\end{equation*}
where $\underline{m}$ is the minimum of medians of $f$, and $\overline{m}$ the maximum of medians of $f$.
\end{dfn}

\begin{prop}[\cite{MMG}*{Section 2.3}]
A sequence $\{X_n\}_{n=1}^\infty$ of mm-spaces is a L\'evy family if and only if for any $f_n \in \Lip_1(X_n)$,
\begin{equation*}
\lim_{n \to \infty} \kf^{m_{X_n}}(f_n, \lm(f_n; m_{X_n})) = 0.
\end{equation*}
\end{prop}

\begin{prop}[\cite{Led}*{Proposition 1.3}]\label{prop:conc_fct_est}
Let $X$ be an mm-space and let $f$ be a $L$-Lipschitz function on $X$. Then we have
\begin{equation}\label{eq:conc_fct_est}
m_X(\left\{x\in X \midd |f(x) - \lm(f;m_X)| \ge \ep \right\}) \le 2\alpha_{X}(\frac{\ep}{L})
\end{equation}
for any $\ep > 0$.
\end{prop}

\subsection{Pyramid}\label{subsec:py}

\begin{dfn}[Pyramid] \label{dfn:py}
A nonempty box-closed subset $\cP \subset \X$ is called a {\it pyramid} if it satisfies the following {\rm(1)} and {\rm(2)}.
\begin{enumerate}
\item If $X \in \cP$ and if $Y \prec X$, then $Y \in \cP$.
\item For any $Y, Y' \in \cP$, there exists $X \in \cP$ such that $Y,Y' \prec X$.
\end{enumerate}
We denote the set of pyramids by $\Pi$.

For an mm-space $X$, we define
\begin{equation*}
\cP X := \left\{Y \in \X \midd Y \prec X \right\},
\end{equation*}
which is a pyramid. We call $\cP X$ the {\it pyramid associated with $X$}.
\end{dfn}

We observe that $Y \prec X$ if and only if $\cP Y \subset \cP X$.
We see that $\{*\}$ and $\X$ are the minimum and maximum pyramids with respect to the inclusion respectively, that is, $\{*\} \subset \cP \subset \X$ for any pyramid $\cP$.

\begin{prop}[\cite{MMG}*{Lemma 6.10}]\label{prop:directed}
Let $\cD$ be a subset of $\X$ satisfying the condition {\rm(2)} of Definition \ref{dfn:py}.
Then
\[
\overline{\bigcup_{X \in \cD} \cP{X}}^{\,\square}
\]
is a pyramid.
\end{prop}

We define the weak convergence of pyramids as follows.
This is also known as the Kuratowski convergence as subsets of $(\X, \square)$.

\begin{dfn}[Weak convergence]
We say that a sequence $\{\cP_n\}_{n=1}^\infty$ of pyramids {\it converges weakly to} a pyramid $\cP$ as $n \to \infty$ if the following (1) and (2) are both satisfied.
\begin{enumerate}
\item For any mm-space $X \in \cP$, we have
\[
\lim_{n \to \infty} \square(X, \cP_n) = 0.
\]
\item For any mm-space $X \in \X \setminus \cP$, we have
\[
\liminf_{n \to \infty} \square(X, \cP_n) > 0.
\]
\end{enumerate}
\end{dfn}

\begin{thm}[\cites{Grmv,MMG}]\label{thm:compactification}
There exists a metric, denoted by $\rho$, on $\Pi$ such that the following {\rm (1) -- (4)} hold.
\begin{enumerate}
\item $\rho$ is compatible with weak convergence.
\item The map
$\iota \colon \X \ni X \mapsto \cP X \in \Pi$
is a $1$-Lipschitz topological embedding map with respect to
$\conc$ and $\rho$.
\item $\Pi$ is $\rho$-compact.
\item $\iota(\X)$ is $\rho$-dense in $\Pi$.
\end{enumerate}
\end{thm}
In particular, $(\Pi,\rho)$ is a compactification of $(\X,\conc)$.
We often identify $X$ with $\cP X$.

\subsection{Dissipation}\label{subsec:dissip}
Dissipation is the opposite notion to concentration.
We omit to state the definition of the infinite dissipation
(see \cite{MMG}*{Definition 8.1} for the definition).
Instead, we state the following proposition.
Let $\{X_n\}_{n=1}^\infty$ be a sequence of mm-spaces.

\begin{prop}[see \cite{MMG}*{Proposition 8.5(2)}]
The sequence $\{X_n\}_{n=1}^\infty$ infinitely dissipates
if and only if $\{\cP X_n\}_{n=1}^\infty$ converges weakly to $\cX$ as $n\to\infty$.
\end{prop}

An easy discussion using \cite{OS}*{Lemma 6.6} leads to the following.

\begin{prop} \label{prop:dissipate}
The sequence $\{X_n\}_{n=1}^\infty$ infinite dissipates if and only if
\[
\lim_{n\to\infty} \OD(X_n;-\kappa) = \infty
\]
for any $\kappa \in (0,1)$.
\end{prop}

\section{Convergence of cones}\label{sec:cone_conv}
We first recall some properties of the cone metric.
For $\kappa \in \R$, we set
\[
I_{\kappa} := \begin{cases}
[0, \pi/\sqrt{\kappa}] & \text{if } \kappa>0, \\
[0,\infty) & \text{if } \kappa \le 0.
\end{cases}
\]
For a metric space $X$, we denote
\[
rx := q(r, x) \in K_\kappa(X), \quad 0 := q(0, x) \in K_\kappa(X)
\]
for $r \in I_\kappa$ and $x \in X$, where $q \colon I_\kappa \times X \to K_\kappa(X)$ is the quotient map.
Moreover, we recall
\[
\sk(r) :=
\begin{cases}
\frac{1}{\sqrt{\kappa}}\sin(\sqrt{\kappa}r) & \text{if } \kappa>0,\\
r & \text{if } \kappa=0,\\
\frac{1}{\sqrt{-\kappa}}\sinh(\sqrt{-\kappa}r) & \text{if } \kappa<0,
\end{cases} \quad
\ck(r) :=
\begin{cases}
\cos(\sqrt{\kappa}r) & \text{if } \kappa>0,\\
1 & \text{if } \kappa=0,\\
\cosh(\sqrt{-\kappa}r) & \text{if } \kappa<0,
\end{cases}
\]
and put
\[
\tsk(r) := \boldsymbol{s}_{\min\{\kappa, 0\}}(r)
\]
for $r \in I_\kappa$.
By the definition of the cone metric \eqref{eq:cone_dist} and the geometry of surface with constant curvature, we obtain the following two lemmas immediately.

\begin{lem}\label{lem:cone_met}
Let $X$ be a metric space and let $\kappa \in \R$.
Then we have
\begin{enumerate}
\item $d_{K_\kappa(X)}(rx, r'x) = |r-r'|$ for any $r,r' \in I_\kappa$ and $x\in X$,
\item $d_{K_\kappa(X)}(rx, rx') \le \tsk(r)\min\{d_X(x,x'),\pi\}$ for any $r \in I_\kappa$ and $x, x' \in X$.
\end{enumerate}
In particular, we have
\begin{equation}
d_{K_\kappa(X)}(rx, r'x') \le |r-r'| + \min\{\tsk(r), \tsk(r')\}\min{\{d_X(x,x'), \pi\}}
\end{equation}
for any $rx, r'x'\in K_\kappa(X)$.
\end{lem}

\begin{lem}\label{lem:cone_met_angle}
Let $X$ and $Y$ be two metric spaces and let $\kappa \in \R$, $\ep > 0$, and $R > 0$.
For any $x, x' \in X$ and $y,y' \in Y$ with
\[
|d_X(x,x') - d_Y(y,y')| < \ep
\]
and for any $r, r' \in I_\kappa \cap [0,R]$, we have
\[
|d_{K_\kappa(X)}(rx, r'x') - d_{K_\kappa(Y)}(ry,r'y')| < \tsk(R)\ep.
\]
\end{lem}

As the first step to prove Theorem \ref{thm:cone_py}, we prove the following.

\begin{thm}\label{thm:cone_box}
Let $\kappa \in \R$.
Assume that a sequence $\{X_n\}_{n=1}^\infty$ of mm-spaces box-converges to an mm-space $X$ and that a sequence $\{\mu_n\}_{n=1}^\infty$ of Borel probability measures on $I_\kappa$ converges weakly to a Borel probability measure $\mu$.
Then $\{K_{\kappa, \mu_n}(X_n)\}_{n=1}^\infty$ box-converges to $K_{\kappa,\mu}(X)$.
\end{thm}

Both the assumption and conclusion of this theorem are stronger than those of Theorem \ref{thm:cone_py}.

\begin{proof}[Proof of Theorem \ref{thm:cone_box}]
Take any small number $\ep > 0$.
Since $\{\mu_n\}_{n=1}^\infty$ converges weakly to $\mu$, there exists $R>0$ such that
\[
\sup_{n\in\N} \mu_n((R, \infty)) < \ep \quad \text{and} \quad \mu((R,\infty)) < \ep.
\]
Note that $R = \pi/\sqrt{\kappa}$ can be assumed if $\kappa>0$.
Since $\{X_n\}_{n=1}^\infty$ box-converges to $X$, there exists an $\ep_n$-mm-isomorphism $f_n \colon X_n \to X$ for each $n$ with $\ep_n \to 0$ by Lemma \ref{lem:ep-mm-isom}.
We define a map $g_n \colon K_{\kappa,\mu_n}(X_n) \to K_{\kappa,\mu}(X)$ by
\[
g_n(rx) := rf_n(x), \quad rx \in K_{\kappa,\mu_n}(X_n).
\]
Then we prove that the map $g_n$ is a $3\ep$-mm-isomorphism from $K_{\kappa,\mu_n}(X_n)$ to $K_{\kappa,\mu}(X)$ for any sufficiently large $n$, which implies that $\square(K_{\kappa,\mu_n}(X_n), K_{\kappa,\mu}(X)) < 9\ep$ by Lemma \ref{lem:ep-mm-isom} again.

Let $\tX_n$ be a non-exceptional domain of $f_n$ and put
\[
\tK_n := q([0,R] \times \tX_n) \subset K_\kappa(X_n) \quad \text{and} \quad \tK := q([0,R] \times X) \subset K_\kappa(X).
\]
Then we have
\[
m_{K_{\kappa,\mu_n}(X_n)}(\tK_n) = \mu_n\otimes m_{X_n}([0,R]\times \tX_n) > (1-\ep)(1-\ep_n) > 1-(\ep+\ep_n).
\]
We consider the metric
\[
d((r,x),(r',x')) := |r-r'| + \tsk(R) \min\{d_{X}(x,x'), \pi\}
\]
on the product space $I_\kappa \times X$.
By \cite{prod}*{Lemma 4.5}, we have
\[
\mu_n \otimes (f_n)_\# m_{X_n}(A) \le \mu \otimes m_X(U_\ep(A)) + \ep
\]
for any Borel subset $A$ on $I_\kappa\times X$
if
\[
2\prok(\mu_n, \mu)+2\tsk(R)  \,\prok((f_n)_\# m_{X_n}, m_X) < \ep.
\]
Moreover we have
\[
U_\ep(q^{-1}(B)) \subset q^{-1}(U_\ep(B))
\]
for any $B \subset \tK$. Indeed, for any $(r, x) \in U_\ep(q^{-1}(B))$, there exists $(r', x') \in q^{-1}(B)$ such that
\[
d((r,x),(r',x')) = |r-r'| + \tsk(R) \min\{d_{X}(x,x'), \pi\} < \ep.
\]
By Lemma \ref{lem:cone_met}, we have $d_{K_\kappa(X)}(rx, r'x') < \ep$.
Therefore, for any Borel subset $B \subset K_{\kappa,\mu}(X)$, we have
\begin{align*}
(g_n)_\# m_{K_{\kappa,\mu_n}(X_n)}(B)
&\le (g_n)_\# m_{K_{\kappa,\mu_n}(X_n)}(B \cap \tK) +\ep +\ep_n \\
&=q_\# (\mu_n\otimes (f_n)_\# m_{X_n})(B \cap \tK) +\ep+\ep_n \\
&\le (\mu\otimes m_X)(U_\ep(q^{-1}(B \cap \tK))) +2\ep+\ep_n \\
&\le (\mu\otimes m_X)(q^{-1}(U_\ep(B \cap \tK))) +2\ep+\ep_n \\
&\le m_{K_{\kappa,\mu}(X)}(U_\ep(B \cap \tK)) +3\ep,
\end{align*}
which implies that $\prok((g_n)_\# m_{K_{\kappa,\mu_n}(X_n)}, m_{K_{\kappa,\mu}(X)}) \le 3\ep$ for any sufficiently large $n$.

For any $rx, r'x' \in \tK_n$,
since
\[
|d_X(f_n(x), f_n(x')) - d_{X_n}(x, x')| \le \ep_n,
\]
Lemma \ref{lem:cone_met_angle} implies that
\[
|d_{K_\kappa(X)}(rf_n(x), r'f_n(x')) - d_{K_\kappa(X_n)}(rx, r'x')| \le \tsk(R)\ep_n <3\ep
\]
for any sufficiently large $n$.
Therefore $g_n$ is a $3\ep$-mm-isomorphism for any sufficiently large $n$ and then $\{K_{\kappa,\mu_n}(X_n)\}_{n=1}^\infty$ box-converges to $K_{\kappa,\mu}(X)$.
\end{proof}

We now define the $\kappa$-cone for a pyramid and prove Theorem \ref{thm:cone_py}.

\begin{dfn}[$\kappa$-Cone for a pyramid]\label{dfn:cone_py}
We define the $\kappa$-cone $K_{\kappa, \mu}(\cP)$ of a pyramid $\cP$ by
\[
K_{\kappa,\mu}(\cP) := \overline{\bigcup_{X \in \cP} \cP{K_{\kappa,\mu}(X)}}^{\, \square}.
\]
\end{dfn}
The $\kappa$-cone $K_{\kappa,\mu}(\cP)$ is a pyramid by Proposition \ref{prop:directed}
since the operation $X \mapsto K_{\kappa,\mu}(X)$ preserves the Lipschitz order.

In order to prove Theorem \ref{thm:cone_py}, we prove the following (1) and (2) under the assumptions.
\begin{enumerate}
\item For any $Y \in K_{\kappa,\mu}(\cP)$, we have $\lim_{n\to\infty} \square(Y, K_{\kappa,\mu_n}(\cP_n)) = 0$.
\item If an mm-space $Y$ satisfies $\liminf_{n\to\infty} \square(Y, K_{\kappa,\mu_n}(\cP_n)) = 0$, then $Y \in K_{\kappa,\mu}(\cP)$.
\end{enumerate}

The proof of (1) of Theorem \ref{thm:cone_py} is an easy work using Theorem \ref{thm:cone_box}.
\begin{proof}[Proof of (1) of Theorem \ref{thm:cone_py}]
Take any $Y \in K_{\kappa,\mu}(\cP)$ and any $\ep>0$.
Then there exist $X \in \cP$ and $Y' \prec K_{\kappa,\mu}(X)$ with $\square(Y, Y') < \ep$.
Since $\cP_n$ converges weakly to $\cP$, there exists a sequence $\{X_n\}_{n=1}^\infty$ with $X_n \in \cP_n$ box-converging to $X$.
By Theorem \ref{thm:cone_box}, the sequence $\{K_{\kappa,\mu_n}(X_n)\}_{n=1}^\infty$ box-converges to $K_{\kappa,\mu}(X)$.
By \cite{MMG}*{Lemma 6.10(1)}, there exists a sequence $\{Y_n\}_{n=1}^\infty$ with $Y_n \prec K_{\kappa,\mu_n}(X_n) \in K_{\kappa,\mu_n}(\cP_n)$ box-converging to $Y'$.
Therefore we obtain
\[
\limsup_{n\to \infty}\square(Y, K_{\kappa,\mu_n}(\cP_n)) < \lim_{n\to \infty}\square(Y', Y_n) + \ep = \ep.
\]
As $\ep \to 0$, we finish the proof.
\end{proof}

The proof of (2) of Theorem \ref{thm:cone_py} is harder than (1).
We mimic the proof of \cite{comts}*{Theorem 1.2}.
We prepare a term and a lemma.

\begin{dfn}[$1$-Lipschitz up to an additive error]
Let $X$ be an mm-space and $Y$ be a metric space. A map $f \colon X \to Y$ is said to be 1-{\it Lipschitz up to} ({\it an additive error}) {\it $\ep \ge 0$} if there exists a Borel subset $\tX \subset X$ such that
\begin{enumerate}
\item $m_X(\tX) \ge 1 - \ep$,
\item $d_Y(f(x), f(x')) \le d_X(x, x') + \ep$ for any $x, x' \in \tX$.
\end{enumerate}
We call $\tX$ a {\it non-exceptional domain} of $f$.
\end{dfn}

\begin{lem}[\cite{comts}*{Corollary 4.7}]\label{lem:error_in_py}
Let $\cP$ be a pyramid and $Y$ an mm-space. Assume that, for any $\ep > 0$, there exist an mm-space $X_\ep \in \cP$ and a Borel measurable map $f_\ep \colon X_\ep \to Y$ such that $f_\ep$ is $1$-Lipschitz up to $\ep$ and $\prok((f_\ep)_\# m_{X_\ep}, m_Y) \le \ep$ holds. Then $Y \in \cP$.
\end{lem}

Moreover, we define a norm $\|\cdot\|_{\infty}$ on $\R^n$ by
\[
\|x\|_\infty := \max_{i} |x_i|, \quad x=(x_1,\ldots,x_n) \in \R^n.
\]

\begin{proof}[Proof of (2) of Theorem \ref{thm:cone_py}]
We first prove the case $\kappa = 0$.
We denote $K_{\mu} = K_{0,\mu}$ in this proof.
Choosing a subsequence, we can assume that
\[
\lim_{n\to\infty} \square(Y, K_{\mu_n}(\cP_n)) = 0.
\]
Then there exist $X_n \in \cP_n$ and $Y_n \prec K_{\mu_n}(X_n)$ for every $n$ such that $\{Y_n\}_{n=1}^\infty$ box-converges to $Y$.
Since $Y_n \prec K_{\mu_n}(X_n)$, there exists a $1$-Lipschitz map $f_n \colon K_{\mu_n}(X_n) \to Y_n$ with ${f_n}_\# m_{K_{\mu_n}(X_n)} = m_{Y_n}$.
Since $\{Y_n\}_{n=1}^\infty$ box-converges to $Y$, there exists an $\ep_n$-mm-isomorphism $g_n\colon Y_n \to Y$ with $\ep_n\to0$. Let $\tY_n$ be a non-exceptional domain of $g_n$.
Here, if $\mu(\{0\}) > 0$, then we can assume that $f_n(0) \in \tY_n$ for any sufficiently large $n$.
Indeed, in this case, we see that $f_n(0) \in U_{\ep}(\tY_n)$ for any fixed $\ep>0$ and any sufficiently large $n$.
By \cite{MMG}*{Lemma 3.5}, there exists a Borel measurable map $\pi_n\colon Y_n \to \tY_n$ such that $\pi_n|_{\tY_n} = \id_{\tY_n}$ and
\[
d_{Y_n}(\pi_n(y), y) \le d_{Y_n}(y, \tY_n) + \ep_n
\]
for any $y \in Y_n$.
Then we see that the map $g_n \circ\pi_n$ is a $(4\ep_n+2\ep)$-mm-isomorphism with non-exceptional domain $U_{\ep}(\tY_n)$.
Thus we can replace $g_n$ with $g_n \circ \pi_n$.

Take any $0<\ep<1$. We find finitely many open subsets $B_1, \ldots, B_N$ in $Y$ such that
\begin{align*}
&\max_{j} \diam{B_j}<\ep, \quad \min_{j} m_Y(B_j)>0, \quad \sum_{j=1}^N m_Y(B_j) > 1-\ep, \\
&\text{and} \quad \delta':= \min_{j\neq j'} d_Y(B_j, B_{j'}) > 0
\end{align*}
(cf.~\cite{OY}*{Lemma 42}).
In the case of $\mu(\{0\}) > 0$, we assume that $g_n \circ f_n(0) \in \bigsqcup_{j=1}^N B_j$. Let
\[
B_0 := Y \setminus \bigsqcup_{j=1}^N B_j
\]
and take a point $y_j \in B_j$ for each $j=0, 1, \ldots, N$.
If $B_0$ is empty, we ignore the case $j=0$.
We define a finite mm-space $\dot{Y}$ as
\[
\dot{Y} := (\{y_0, y_1, \ldots, y_N\}, d_Y, \sum_{j=0}^N \mu_Y(B_j)\delta_{y_j}).
\]
Note that the inclusion map $\iota\colon \dot{Y} \to Y$ is an $\ep$-mm-isomorphism.

Since $\{\mu_n\}_{n=1}^\infty$ converges weakly to $\mu$, there exists $R > 1$ such that $\mu(\{R\})=0$ and
\[
\mu((R, \infty)) \le \sup_{n \in \N} \mu_n((R, \infty)) < \frac{\ep}{N}.
\]
Take any sufficiently small number $\eta > 0$ as $2\eta < \min\{\delta', \ep\}$.
We find finitely many disjoint intervals $I_1, \ldots, I_M \subset [0,R]$ such that
\begin{align*}
&\max_{i} \diam{I_i}<\eta, \quad \min_{i} \mu(I_i)>0, \quad \sum_{i=1}^M \mu(I_i) = \mu([0,R]), \quad \text{and} \quad \max_{i}\mu(\partial I_i) = 0.
\end{align*}
Here, $\max_{i}\mu(\partial I_i) = 0$ follows from the fact that the set of atomic points of $\mu$ is countable.
Note that $\mu_n(I_i) \to \mu(I_i)$ as $n\to\infty$ since $\mu(\partial I_i) = 0$.
Moreover, we assume that $0 \not\in \bigsqcup_{i=1}^M I_i$ if $\mu(\{0\}) = 0$.
If $\mu(\{0\}) > 0$, then there exist unique numbers $i \in\{1,\ldots,M\}$ and $j \in\{1,\ldots,N\}$ with $0 \in I_i$ and $f_n(0) \in g_n^{-1}(B_j)\cap\tY_n$.
We denote them by $i_0$ and $j_0$.

We define subsets $A_{ij}^n$ of $X_n$ for $i=1,\ldots,M$ and $j=1,\ldots,N$ by
\[
A_{ij}^n := \left\{x \in X_n \midd \text{there exists } s\in I_i \text{ such that } sx \in f_n^{-1}(g_n^{-1}(B_j) \cap \tY_n)\right\}.
\]
Note that $A_{i_0j_0}^n = X_n$ if $\mu(\{0\}) > 0$ and $n$ is large.
We define a 1-Lipschitz map $\Phi_n \colon X_n \to (\R^{MN}, \|\cdot\|_\infty)$ by
\[
\Phi_n(x) := (\min\{d_{X_n}(x, A_{ij}^n), \pi\})_{i=1,\ldots,M, \, j=1,\ldots,N}.
\]
Then, since $\supp{(\Phi_n)_\# m_{X_n}}$ is in the compact set $\left\{z\in \R^{MN} \midd \|z\|_\infty \le \pi\right\}$, the sequence $\{(\Phi_n)_\# m_{X_n}\}_{n=1}^\infty$ is tight.
Thus there exists a subsequence of $\{(\Phi_n)_\# m_{X_n}\}_{n=1}^\infty$ converging weakly to a Borel probability measure $\nu$ on $\left\{z\in \R^{MN} \midd \|z\|_\infty \le \pi\right\}$.
Defining an mm-space $X$ by
\[
X := (\supp{\nu}, \|\cdot\|_{\infty}, \nu),
\]
we have $X \in \cP$ since $(\supp{(\Phi_n)_\# m_{X_n}}, \|\cdot\|_{\infty}, (\Phi_n)_\# m_{X_n}) \in \cP_n$ and $\{\cP_n\}_{n=1}^\infty$ converges weakly to $\cP$.

Now our goal is to find a Borel measurable map $\Psi \colon K_\mu(X) \to \dot{Y}$ such that
$\Psi$ is 1-Lipschitz up to $3\ep$ and $\prok(\Psi_\# m_{K_\mu(X)}, m_{\dot{Y}}) \le 2\ep$.
If we prove this, then we see that the map $\iota \circ \Psi \colon K_\mu(X) \to Y$ is 1-Lipschitz up to $4\ep$ and satisfies $\prok((\iota\circ\Psi)_\# m_{K_\mu(X)}, m_Y) \le 3\ep$,
which leads to $Y \in K_\mu(\cP)$ by Lemma \ref{lem:error_in_py}.

For any $i \in \{1,\ldots,M\}$, we define a map $\proj_i \colon \R^{MN}\to\R^N$ by
\[
\proj_i((z_{ij})_{i=1,\ldots,M,\,j=1,\ldots,N}) := (z_{i1}, \ldots, z_{iN}), \quad (z_{ij})_{i=1,\ldots,M,\,j=1,\ldots,N} \in \R^{MN}
\]
and put $\nu_i := (\proj_i)_\# \nu$ and $\Phi_{n,i} := \proj_i \circ \Phi_n$.
Note that $(\Phi_{n,i})_\# m_{X_n}$ converges weakly to $\nu_i$ as $n\to\infty$.
Moreover, we define a disjoint family $\{Z_j\}_{j=1}^N$ of subsets of $\R^N$ as
\[
Z_j := \left\{(z_1, \ldots, z_N) \midd z_j=0 \text{ and } \min_{j'\neq j}z_{j'} \ge \delta \right\}
\]
for a fixed positive real number $\delta > 0$ with
\[
\delta < \min\{\frac{\delta'}{3R}, \pi\}.
\]

\begin{clm}\label{clm:Phi_Z}
For every $i=1,\ldots,M$ and every $j=1,\ldots,N$, we have
\[
\Phi_{n,i}(A_{ij}^n) \subset Z_j
\]
if $n$ is sufficiently large.
\end{clm}

\begin{proof}
Take any $x \in A_{ij}^n$ and let $(z_1, \ldots, z_N) = \Phi_{n,i}(x)$.
It is obvious that
\[
z_j = \min\{d_{X_n}(x, A_{ij}^n), \pi\} = 0.
\]
We prove that $z_{j'} = \min\{d_{X_n}(x, A_{ij'}^n), \pi\} \ge \delta$ if $j' \neq j$.
Take any $x' \in A_{ij'}^n$.
There exist $r, r' \in I_i$ such that
\[
rx \in f_n^{-1}(g_n^{-1}(B_j)\cap \tY_n) \quad \text{and} \quad r'x' \in f_n^{-1}(g_n^{-1}(B_{j'})\cap \tY_n).
\]
Then, by Lemma \ref{lem:cone_met}, we have
\begin{align*}
\delta' &\le d_Y(B_j, B_{j'}) \le d_Y(g_n\circ f_n(rx), g_n\circ f_n(r'x')) \\
&\le d_{K_0(X_n)}(rx, r'x') +\ep_n \le |r-r'| + Rd_{X_n}(x,x') + \ep_n \\
&\le \eta + Rd_{X_n}(x, x') + \ep_n,
\end{align*}
which implies that $d_{X_n}(x, x') \ge \delta$ if $n$ is sufficiently large. Thus we obtain $z_{j'} \ge \delta$. The proof is completed.
\end{proof}

By this claim and the portmanteau theorem of weak convergence of measures, we have
\[
\nu_i(Z_j) \ge \limsup_{n\to\infty} (\Phi_{n,i})_\# m_{X_n}(Z_j) \ge \limsup_{n\to \infty} m_{X_n}(A_{ij}^n)
\]
for any $j=1,\ldots,N$.
Now we define a Borel measurable map $\Psi\colon K_{\mu}(X) \to \dot{Y}$ by
\[
\Psi(rx) :=
\begin{cases}
y_j & \text{if } r \in I_i \text{ and } \proj_i(x) \in Z_j, \\
y_0 & \text{otherwise}
\end{cases}
\]
for $rx \in K_\mu(X)$.
Here, if $B_0 =\emptyset$, then we replace $y_0$ with an arbitrary point of $\{y_1,\ldots,y_N\}$.
Note that if $\mu(\{0\})>0$, then we have $\nu_{i_0}(Z_{j_0}) = 1$ and then
\[
X = \supp{\nu} \subset \proj_{i_0}^{-1}(Z_{j_0}),
\]
which implies that $\Psi(0) = y_{j_0}$. Thus $\Psi$ is well-defined on $K_\mu(X)$.
From now on, we denote $d_{K_n} := d_{K_0(X_n)}$, $d_K := d_{K_0(X)}$, $m_{K_n} := m_{K_{\mu_n}(X_n)}$, and $m_K := m_{K_\mu(X)}$ in this proof for simplicity.

We first prove that $\prok(\Psi_\# m_K, m_{\dot{Y}}) \le 2\ep$.
For any $j=1,\ldots,N$, we have
\begin{align*}
m_{\dot{Y}}(\{y_j\}) &= m_Y(B_j) \le \liminf_{n\to\infty} (g_n)_\# m_{Y_n}(B_j) \\
&\le \liminf_{n\to\infty} (m_{Y_n}(g_n^{-1}(B_j)\cap \tY_n) + \ep_n)
\le \limsup_{n\to\infty} m_{K_n}(f_n^{-1}(g_n^{-1}(B_j)\cap \tY_n)) \\
&= \limsup_{n\to\infty} \mu_n \otimes m_{X_n}(q^{-1}(f_n^{-1}(g_n^{-1}(B_j)\cap \tY_n))) \\
&\le \limsup_{n\to\infty} \sum_{i=1}^M \mu_n \otimes m_{X_n}(q^{-1}(f_n^{-1}(g_n^{-1}(B_j)\cap \tY_n)) \cap (I_i \times X_n)) +\mu_n((R,\infty)) \\
&\le \limsup_{n\to\infty} \sum_{i=1}^M \mu_n(I_i) \, m_{X_n}(A_{ij}^n) +\frac{\ep}{N}
\le \sum_{i=1}^M \mu(I_i) \, \nu_i(Z_j) +\frac{\ep}{N}.
\end{align*}
On the other hand, we have
\begin{align*}
\Psi_\# m_K(\{y_j\}) &= m_K(\bigsqcup_{i=1}^M q(I_i \times \proj_i^{-1}(Z_j))) \\
&\ge \mu \otimes \nu(\bigsqcup_{i=1}^M I_i \times \proj_i^{-1}(Z_j)) = \sum_{i=1}^M \mu(I_i)\nu_i(Z_j).
\end{align*}
These imply that
\[
m_{\dot{Y}}(\{y_j\}) \le \Psi_\# m_K(\{y_j\}) + \frac{\ep}{N}
\]
for any $j=1,\ldots,N$.
Taking $m_{\dot{Y}}(\{y_0\}) < \ep$ into account, we obtain
\[
\prok(\Psi_\# m_K, m_{\dot{Y}}) < \ep+N\cdot\frac{\ep}{N} = 2\ep.
\]

We next prove that $\Psi$ is 1-Lipschitz up to $3\ep$.
Letting
\[
\tK := q(\bigsqcup_{i=1}^M (I_i \times \proj_i^{-1}(\supp{\nu_i} \cap \bigsqcup_{j=1}^N Z_j))),
\]
we have
\[
m_K(\tK) \ge \sum_{i=1}^M \sum_{j=1}^N \mu(I_i) \nu_i(Z_j) \ge \sum_{j=1}^N(m_{\dot{Y}}(\{y_j\})-\frac{\ep}{N}) \ge 1-2\ep.
\]
Take any $sx, s'x' \in \tK$ and fix them.
There exist $i, i' \in \{1,\ldots,M\}$ and $j,j' \in \{1,\ldots,N\}$ such that
\[
s \in I_i, \quad \proj_i(x) \in Z_j, \quad \text{and} \quad s' \in I_{i'}, \quad  \proj_{i'}(x') \in Z_{j'}.
\]
Then, since $(\Phi_n)_\# m_{X_n}$ converges weakly to $\nu$, there exist $x_n \in A_{ij}^n$ and $x'_n \in A_{i'j'}^n$ for each $n$ such that
\[
\|\Phi_n(x_n) - x\|_\infty, \|\Phi_n(x'_n) -x'\|_\infty \to 0 \text{ as } n\to\infty.
\]
We prove
\begin{equation}\label{eq:Psi_1-Lip_upto}
d_Y(y_j, y_{j'}) \le \sqrt{s^2+(s')^2-2ss'\cos(\min\{d_{X_n}(x_n, A_{i'j'}^n), \pi\})} +\ep_n+ 3\ep.
\end{equation}
Indeed, taking any $\tilde{x}' \in A_{i'j'}^n$, there exist $t \in I_i$ and $t' \in I_{i'}$ such that
\[
tx_n \in f_n^{-1}(g_n^{-1}(B_j) \cap \tY_n) \quad \text{and} \quad t'\tilde{x}' \in f_n^{-1}(g_n^{-1}(B_{j'}) \cap \tY_n).
\]
Thus we have
\begin{align*}
d_Y(y_j, y_{j'}) &\le d_Y(B_j, B_{j'}) +2\ep \le d_Y(g_n\circ f_n(tx_n), g_n\circ f_n(t'\tilde{x}')) + 2\ep \\
&\le d_{K_n}(tx_n, t'\tilde{x}') + \ep_n + 2\ep \le d_{K_n}(sx_n, s'\tilde{x}') + |s-t| + |s'-t'| + \ep_n +2\ep\\
&< d_{K_n}(sx_n, s'\tilde{x}') + 2\eta + \ep_n +2\ep < d_{K_n}(sx_n, s'\tilde{x}') + \ep_n +3\ep.
\end{align*}
Therefore we obtain \eqref{eq:Psi_1-Lip_upto}.
Moreover, since
\begin{align*}
\min\{d_{X_n}(x_n, A_{i'j'}^n), \pi\} &= \min\{d_{X_n}(x_n, A_{i'j'}^n), \pi\} - \min\{d_{X_n}(x'_n, A_{i'j'}^n), \pi\} \\
&\le \|\Phi_n(x_n)-\Phi_n(x'_n)\|_\infty,
\end{align*}
we have
\begin{align*}
d_Y(y_j, y_{j'}) &\le \lim_{n\to\infty} \sqrt{s^2+(s')^2-2ss'\cos(\|\Phi_n(x_n)-\Phi_n(x'_n)\|_\infty)} +\ep_n+ 3\ep \\
&= \sqrt{s^2+(s')^2-2ss'\cos(\|x-x'\|_\infty)} + 3\ep \\
& = d_K(sx, s'x') + 3\ep.
\end{align*}
Therefore $\Psi$ is 1-Lipschitz up to $3\ep$.

We apply Lemma \ref{lem:error_in_py} to the map $\iota\circ\Psi$ and then obtain $Y \in K_\mu(\cP)$.
This completes the proof of the case $\kappa = 0$.

In the case of $\kappa<0$, taking
\[
\delta < \min\{\frac{\delta'}{3\sk(R)}, \pi\},
\]
we obtain Claim \ref{clm:Phi_Z} by Lemma \ref{lem:cone_met} and
\[
d_Y(y_j, y_{j'}) \le \ck^{-1}(\ck(s)\ck(s')+\kappa\sk(s)\sk(s')\cos(\min\{d_{X_n}(x_n, A_{i'j'}^n), \pi\})) +\ep_n+ 3\ep
\]
instead of \eqref{eq:Psi_1-Lip_upto}.
Thus we can prove that $\Psi$ is 1-Lipschitz up to $3\ep$ and $\prok(\Psi_\# m_{K_{\kappa, \mu}(X)}, m_{\dot{Y}}) \le 2\ep$.
In the case of $\kappa > 0$, we have to pay attention to the handling of the two points $0, \pi/\sqrt{\kappa} \in K_\kappa(X)$.
Eventually, it can be proved similarly by Lemma \ref{lem:cone_met}.
We now finish the proof of Theorem \ref{thm:cone_py}.
\end{proof}

\section{Another simple proof for L\'evy families}\label{sec:main_proof}

In this section, we give another simple proof to Theorem \ref{thm:main}.

\begin{prop}\label{prop:f_r}
Let $f$ be a $1$-Lipschitz function on the $\kappa$-cone of a metric space $X$.
For any $r \in I_\kappa$, we define a function $f_r\colon X \to \R$ by
\begin{equation}\label{eq:f_r}
f_r(x) := f(rx), \quad x \in X.
\end{equation}
Then $f_r$ is $\tsk(r)$-Lipschitz with respect to $\min\{d_X, \pi\}$.
\end{prop}

\begin{proof}
This proposition follows from Lemma \ref{lem:cone_met} directly.
\end{proof}

\begin{prop}[cf.~\cite{MMG}*{Proof of Proposition 7.32}]\label{prop:lm_1-Lip}
Let $f$ be a $1$-Lipschitz function on the $\kappa$-cone of an mm-space $X$.
We define a function $\bar{f} \colon I_\kappa \to \R$ by
\begin{equation}\label{eq:bar_f}
\bar{f}(r) := \lm(f_r; m_X), \quad r \in I_\kappa,
\end{equation}
where $f_r$ is defined as \eqref{eq:f_r} for $r \in I_\kappa$.
Then the function $\bar{f}$ is $1$-Lipschitz on $I_\kappa$.
\end{prop}

\begin{proof}
We denote by $\underline{m}(r)$ and $\overline{m}(r)$ the minimum and maximum of medians of $f_r$ with respect to $m_X$, respectively.
Note that
\[
\bar{f}(r) = \frac{\underline{m}(r) + \overline{m}(r)}{2}.
\]
Then, for any $r,r' \in I_\kappa$, since
\[
|f_r(x)-f_{r'}(x)| \le d_{K_\kappa(X)}(rx, r'x) = |r-r'|,
\]
we have
\[
m_X(\left\{x \in X \midd f_{r'}(x) \le \underline{m}(r) + |r-r'|\right\}) \ge m_X(\left\{x \in X \midd f_r(x) \le \underline{m}(r)\right\}) \ge \frac{1}{2},
\]
which together with the minimality of $\underline{m}(r')$ implies
\[
\underline{m}(r') \le \underline{m}(r) + |r-r'|.
\]
Exchanging $r$ and $r'$, we obtain
\[
|\underline{m}(r) - \underline{m}(r')| \le |r-r'|.
\]
Thus $\underline{m}$ is $1$-Lipschitz on $I_\kappa$.
Similarly, $\overline{m}$ is also 1-Lipschitz on $I_\kappa$, and then $\bar{f}$ is 1-Lipschitz on $I_\kappa$.
The proof is completed.
\end{proof}

\begin{proof}[Proof of Theorem \ref{thm:main}]
For any $f \in \Lip_1(K_\kappa(X_n))$, the functions $f_r\colon X_n \to \R$, $r \in I_\kappa$, and $\bar{f}\colon I_\kappa \to \R$ are defined as \eqref{eq:f_r} and \eqref{eq:bar_f}, respectively.
Propositions \ref{prop:f_r} and \ref{prop:lm_1-Lip} imply that $f_r$ is $\tsk(r)$-Lipschitz with respect to $\min\{d_{X_n}, \pi\}$ and $\bar{f}$ is $1$-Lipschitz on $I_\kappa$. Moreover we have
\[
|f_r(x) - \bar{f}(r)| \le \diam{f_r(X_n)} \le \pi \tsk(r)
\]
for any $r\in I$ and $x \in X_n$.

In order to prove Theorem \ref{thm:main}, by Lemma \ref{lem:enforce} and the assumption, it is sufficient to prove that,
for any $\ep>0$, there exists a number $N\in\N$ such that any $f \in \Lip_1(K_\kappa(X_n))$ satisfies
\begin{equation}\label{eq:main_proof}
m_{K_{\kappa,\mu_n}(X_n)}(\left\{rx \in K_\kappa(X_n) \midd |f(rx) - \bar{f}(r)| \ge \ep\right\}) \le \ep
\end{equation}
if $n \ge N$.
This together with the $1$-Lipschitz continuity of $p_n(rx):= r$ on $K_\kappa(X_n)$ implies
\[
\haus(\Lip_1(K_\kappa(X_n)), (p_n)^*\Lip_1(I_\kappa)) \le \ep,
\]
where the Hausdorff distance is defined with respect to $\kf^{m_{K_{\kappa, \mu_n}(X_n)}}$.

Since $\{\mu_n\}_{n=1}^\infty$ converges weakly, there exists $R > 0$ such that
\[
\sup_{n\in\N} \mu_n((R,\infty)) < \ep.
\]
We assume that $R = \pi/\sqrt{\kappa}$ if $\kappa>0$.
Thus we have
\begin{align*}
&m_{K_{\kappa, \mu_n}(X_n)}(\left\{rx \in K_{\kappa}(X_n) \midd |f(rx) - \bar{f}(r)| \ge \ep\right\}) \\
&\le m_{K_{\kappa, \mu_n}(X_n)}(\left\{rx \in K_{\kappa}(X_n) \midd |f(rx) - \bar{f}(r)| \ge \ep \text{ and } r \le R\right\}) + \ep.
\end{align*}
Here, by the Markov inequality, we calculate
\begin{align*}
&m_{K_{\kappa, \mu_n}(X_n)}(\left\{rx \in K_{\kappa}(X_n) \midd |f(rx) - \bar{f}(r)| \ge \ep \text{ and } r \le R\right\})\\
&=m_{K_{\kappa, \mu_n}(X_n)}(\left\{rx \in K_{\kappa}(X_n) \midd \mathbf{1}_{\{r\le R\}}(rx)|f(rx) - \bar{f}(r)| \ge \ep\right\})\\
&\le \frac{1}{\ep} \int_{q([0,R] \times X_n)} |f(rx)-\bar{f}(r)| \, dm_{K_{\kappa,\mu_n}(X_n)}(rx) \\
&= \frac{1}{\ep} \int_0^R \left(\int_{X_n}\left|f_r(x)-\bar{f}(r)\right| \, dm_{X_n}(x)\right) d\mu_n(r).
\end{align*}
By Proposition \ref{prop:conc_fct_est}, we have
\[
m_{X_n}\left(\left\{x \in X_n\midd |f_r(x) - \bar{f}(r)| \ge \ep^2 \right\}\right)
\le 2\alpha_{X_n}(\frac{\ep^2}{\tsk(r)}).
\]
Thus we estimate
\begin{align*}
&\int_{X_n} \left|f_r- \bar{f}(r)\right| \, dm_{X_n} \\
&= \int_{|f_r-\bar{f}(r)|\le\ep^2} \left|f_r-\bar{f}(r) \right| \, dm_{X_n} + \int_{|f_r-\bar{f}(r)|>\ep^2} \left|f_r-\bar{f}(r) \right| \, dm_{X_n} \\
&\le \ep^2 + \pi \tsk(r) \, m_{X_n}(\{|f_r - \bar{f}(r)| > \ep^2 \}) \le \ep^2 + 2\pi \tsk(r) \,\alpha_{X_n}(\frac{\ep^2}{\tsk(r)}).
\end{align*}
These imply that
\begin{align*}
&m_{K_{\kappa,\mu_n}(X_n)}(\left\{rx \in K_\kappa(X_n) \midd |f(rx) - \bar{f}(r)| \ge \ep\right\}) \\
&\le \frac{1}{\ep}
\int_0^R \left(\ep^2 + 2\pi \tsk(r) \, \alpha_{X_n}(\frac{\ep^2}{\tsk(r)})\right) d\mu_n(r) + \ep \\
&\le 2\ep + \frac{2\pi}{\ep}
\int_0^R \alpha_{X_n}(\frac{\ep^2}{\tsk(r)}) \, \tsk(r) d\mu_n(r)\\
&\le 2\ep + \frac{2\pi \tsk(R)}{\ep}
\alpha_{X_n}(\frac{\ep^2}{\tsk(R)}),
\end{align*}
where the last inequality follows from the fact that  $r \mapsto \alpha_{X_n}(\ep^2/r)$ is non-decreasing.
Since $\{X_n\}_{n=1}^\infty$ is a L\'evy family, the concentration function $\alpha_{X_n}$ converges pointwise to $0$ as $n\to\infty$. Thus, we obtain
\[
\limsup_{n\to\infty}
m_{K_{\kappa,\mu_n}(X_n)}(\left\{rx \in K_\kappa(X_n) \midd |f(rx) - \bar{f}(r)| \ge \ep\right\}) \le 2\ep.
\]
Therefore we obtain \eqref{eq:main_proof}. We finish the proof of Theorem \ref{thm:main}.
\end{proof}

\begin{rem}
By \cite{Led}*{Proposition 1.8}, \eqref{eq:main_proof} also holds if $\bar{f}$ is  replaced with
\[
\bar{f}(r) := \int_{X_n} f(rx) \, dm_{X_n}(x).
\]
\end{rem}

We also obtain the following corollary about the Hopf action.
Let $F = \R$, $\C$, or $\Hb$, where $\Hb$ is the algebra of quaternions, and let $d:= \dim_\R{F}$.
The multiplicative group
\[
U^F := \left\{t \in F \midd \|t\| = 1\right \}
\]
acts on the unit sphere $S^{dn-1}$ in $F^n$ by
\[
S^{dn-1} \times U^F \ni ((x_1, \ldots, x_n), t) \mapsto (tx_1, \ldots, tx_n) \in S^{dn-1},
\]
where $(x_1, \ldots, x_n) \in S^{dn-1} \subset F^n$ with $x_i \in F$.
This action extends to the cone $K_\kappa(S^{dn-1})$ as
\[
K_\kappa(S^{dn-1}) \times U^F \ni (rx, t) \mapsto r(tx) \in K_\kappa(S^{dn-1})
\]
and it is isometric and preserves the measure $m_{K_{\kappa,\mu}(S^{dn-1})}$ for any $\mu$, that is, the mm-action.
Thus we have the quotient mm-space
\[
K_{\kappa, \mu}(S^{dn-1})/U^F
\]
but, after all, this coincides with the cone of the quotient of spheres
\[
K_{\kappa, \mu}(S^{dn-1}/U^F).
\]

\begin{cor}\label{cor:Hopf}
Let $F = \R$, $\C$, or $\Hb$ and let $d:= \dim_\R{F}$ and $\kappa \in \R$.
Assume that a sequence $\{\mu_n\}_{n=1}^\infty$ of Borel probability measures on $I_\kappa$ converges weakly to a Borel probability measure $\mu$.
Then the sequence $\{K_{\kappa,\mu_n}(S^{dn-1})/U^F\}_{n=1}^\infty$ concentrates to the interval $(I_\kappa, |\cdot|, \mu)$.
\end{cor}

\begin{proof}
This directly follows from Theorem \ref{thm:main} since
\[
K_{\kappa,\mu_n}(S^{dn-1})/U^F = K_{\kappa,\mu_n}(S^{dn-1}/U^F)
\]
and $\{S^{dn-1}/U^F\}_{n=1}^\infty$ is a L\'evy family.
\end{proof}

In this type of concentration, the limit space is independent of the Hopf action.

\begin{rem}
In \cite{NS}, the concentration with group action, called the {\it equivariant concentration}, has been studied.
The equivariant concentration implies the concentration of the quotient spaces.
Actually, we can prove that the pair $(K_{\kappa,\mu_n}(S^{dn-1}), U^F)$ equivariant concentrates to $((I_\kappa, |\cdot|, \mu), \{\id_{I_\kappa}\})$ from our proof of Theorem \ref{thm:main}.
\end{rem}

\section{Cauchy distribution}\label{sec:cauchy}
In this section, we deal with the Cauchy distribution and prove Theorems \ref{thm:main_cauchy} and \ref{thm:main_cauchy_gene} by applying Theorem \ref{thm:main}.
We recall the Cauchy measure $\nu^n_\beta$ on $\R^n$ as in \eqref{eq:Cauchy};
\[
\nu^n_\beta(dx) := \frac{\Gamma(\frac{n+\beta}{2})}{\pi^{\frac{n}{2}}\Gamma(\frac{\beta}{2})}\frac{dx}{(1+\|x\|^2)^\frac{n+\beta}{2}},
\]
and the Borel probability measure $\nu_\beta$ on $[0,\infty)$ as in \eqref{eq:recip};
\[
\nu_\beta(dt) := \frac{2^{1-\frac{\beta}{2}}}{\Gamma(\frac{\beta}{2})} \frac{1}{t^{\beta+1}} e^{-\frac{1}{2t^2}} \, dt.
\]
We denote $C^n_\beta := (\R^n, \|\cdot\|, \nu^n_\beta)$ and call it the $n$-dimensional Cauchy space.

\begin{prop}\label{prop:Cauchy_dom}
Let $\pi^n_k \colon \R^n \to \R^k$, $n>k$, be the natural projection given by
\[
\pi^n_k(x_1,\ldots,x_n) = (x_1, \ldots, x_k).
\]
Then we have
\begin{equation}\label{eq:consistency}
(\pi^n_k)_\# \nu^n_\beta = \nu^k_\beta.
\end{equation}
In particular, we have
\[
C^1_\beta \prec C^2_\beta \prec \cdots \prec C^n_\beta \prec \cdots.
\]
\end{prop}

\begin{proof}
It is sufficient to prove the case of $k=n-1$.
Take any Borel measurable function $f$ on $\R^{n-1}$. Changing of variable
\[
x_n = (1 + x_1^2 + \cdots + x_{n-1}^2)^{\frac{1}{2}} y,
\]
we have
\begin{align*}
\int_{\R^{n-1}} f \, d(\pi^n_{n-1})_\#\nu^{n}_\beta &=
\frac{\Gamma(\frac{n+\beta}{2})}{\pi^{\frac{n}{2}}\Gamma(\frac{\beta}{2})}\int_{\R^n} \frac{f(x_1,\ldots, x_{n-1})}{(1+x_1^2+\cdots+x_n^2)^\frac{n+\beta}{2}} \, dx_1\cdots dx_n \\
&= \frac{\Gamma(\frac{n+\beta}{2})}{\pi^{\frac{n}{2}}\Gamma(\frac{\beta}{2})}
\int_{-\infty}^\infty  \frac{dy}{(1+y^2)^{\frac{n+\beta}{2}}} \int_{\R^{n-1}} \frac{f(x_1,\ldots, x_{n-1})\, dx_1\cdots dx_{n-1}}{(1+x_1^2+\cdots+x_{n-1}^2)^{\frac{n-1+\beta}{2}}} \\
&= \frac{\Gamma(\frac{n+\beta}{2})}{\pi^{\frac{n}{2}}\Gamma(\frac{\beta}{2})}
\frac{\sqrt{\pi}\,\Gamma(\frac{n-1+\beta}{2})}{\Gamma(\frac{n+\beta}{2})} \int_{\R^{n-1}} \frac{f(x_1,\ldots, x_{n-1})\, dx_1\cdots dx_{n-1}}{(1+x_1^2+\cdots+x_{n-1}^2)^{\frac{n-1+\beta}{2}}} \\
&= \int_{\R^{n-1}} f \, d\nu^{n-1}_\beta.
\end{align*}
Therefore we obtain $(\pi^n_{n-1})_\# \nu^n_\beta = \nu^{n-1}_\beta$ and then $C^{n-1}_\beta \prec C^n_\beta$.
\end{proof}

For $\beta > 0$, the Borel probability measure $\tilde{\nu}^n_\beta$ on $\R^n$ is defined by
\begin{equation}\label{eq:scaled_Cauchy}
\tilde{\nu}^n_\beta(dx) := \frac{n^{\frac{n}{2}}\Gamma(\frac{n+\beta}{2})}{\pi^{\frac{n}{2}}\Gamma(\frac{\beta}{2})}\frac{dx}{(1+n\|x\|^2)^\frac{n+\beta}{2}}.
\end{equation}

\begin{prop}\label{prop:scale}
Let $\tC^n_\beta := (\R^n, \|\cdot\|, \tilde{\nu}^n_\beta)$.
Then $\frac{1}{\sqrt{n}}C^n_\beta$ is mm-isomorphic to $\tC^n_\beta$.
\end{prop}

\begin{proof}
We define a map $L_n \colon \R^n \to \R^n$ by
\[
L_n(x) = \frac{x}{\sqrt{n}}, \quad x \in \R^n.
\]
The map $L_n$ is an isometry from $\frac{1}{\sqrt{n}} C^n_\beta$ to $\tC^n_\beta$ and $(L_n)_\# \nu^n_\beta = \tilde{\nu}^n_\beta$ holds.
Therefore $\frac{1}{\sqrt{n}}C^n_\beta$ and $\tC^n_\beta$ are mm-isomorphic to each other.
\end{proof}

The following is a key lemma to apply Theorem \ref{thm:main} to the Cauchy distribution.

\begin{lem}\label{lem:rad_conv}
Let $p_n \colon \R^n \to [0, \infty)$ be the function defined by
\[
p_n(x) := \|x\|, \quad x\in \R^n.
\]
Then $(p_n)_\# \tilde{\nu}^n_\beta$ converges weakly to the Borel probability measure $\nu_\beta$ as $n\to\infty$.
\end{lem}

\begin{proof}
Take any bounded {\rm (}not necessarily continuous{\rm )} function $f$ on $[0,\infty)$.
By the polar transformation, we have
\begin{equation} \label{eq:exradcau1}
\int_{0}^{\infty} f(t) \, d(p_n)_{\#} \tilde{\nu}^n_\beta(t)
= \int_{0}^{\infty} \frac{2\Gamma(\frac{n+\beta}{2})}{\Gamma(\frac{n}{2})\Gamma(\frac{\beta}{2})} \frac{f(t)}{t(1+nt^2)^{\frac{\beta}{2}}} \left(\frac{nt^2}{1+nt^2}\right)^{\frac{n}{2}} \, dt.
\end{equation}
By Stirling's formula, we have
\[
\lim_{n \to \infty} \frac{\Gamma\left( \frac{n+\beta}{2} \right)}{\Gamma\left( \frac{n}{2} \right)} \frac{1}{t(1+nt^2)^{\frac{\beta}{2}}}\left( \frac{nt^2}{1+nt^2} \right)^{\frac{n}{2}} = \frac{1}{2^{\frac{\beta}{2}}t^{\beta+1}}e^{-\frac{1}{2t^2}}
\]
for any $t> 0$.
Moreover, since
\[
\sup_{n\in\N}\frac{\Gamma(\frac{n+\beta}{2})}{\Gamma(\frac{n}{2}) \, n^{\frac{\beta}{2}}} < \infty
\quad \text{and} \quad
0 < \frac{n^{\frac{\beta}{2}}}{t(1+nt^2)^{\frac{\beta}{2}}} \left( \frac{nt^2}{1+nt^2} \right)^{\frac{n}{2}} < \frac{1}{t^{\beta+1}}
\]
for any $t>0$, there exists a constant $C > 0$ depending only on $f$ such that
\[
\left|\frac{2\Gamma(\frac{n+\beta}{2})}{\Gamma(\frac{n}{2})\Gamma(\frac{\beta}{2})} \frac{f(t)}{t(1+nt^2)^{\frac{\beta}{2}}} \left(\frac{nt^2}{1+nt^2}\right)^{\frac{n}{2}}\right| < C \min\{1, \frac{1}{t^{\beta+1}}\}.
\]
Therefore, taking the limit of \eqref{eq:exradcau1} by the dominated convergence theorem, we obtain this lemma.
\end{proof}

\begin{proof}[Proof of Theorem \ref{thm:main_cauchy}]
Let $S^{n-1}$ be the unit sphere in $\R^n$ centered at origin.
Assume that $S^{n-1}$ is equipped with the Riemannian distance and the uniform probability measure $\sigma^{n-1}$.
It is well-known that the Euclidean space $(\R^n, \|\cdot\|)$ is isometric to the Euclidean cone $K_0(S^{n-1})$.
Moreover, by the polar transformation, we have
\[
\tilde{\nu}^n_\beta = q_\#\left((p_n)_\# \tilde{\nu}^n_\beta \otimes \sigma^{n-1}\right).
\]
Thus we obtain
\[
\tC^n_\beta = K_{0,(p_n)_\# \tilde{\nu}^n_\beta}(S^{n-1}).
\]
Therefore, by Theorem \ref{thm:main} together with Lemma \ref{lem:rad_conv} and \eqref{eq:conc_sphere},
the sequence $\{\tC^n_\beta\}_{n=1}^\infty$
concentrates to the half line $([0,\infty), |\cdot|, \nu_\beta)$.
This completes the proof.
\end{proof}

\begin{rem}
In the case of $\R^n = K_0(S^{n-1})$, Theorem \ref{thm:main} is also expressed as follows.
Let $\hat{\mu}_n$ be a radially symmetric Borel probability measure on $\R^n$ and let $p_n(x) := \|x\|$ for $x \in \R^n$. Assume that the measure $\mu_n := (p_n)_\# \hat{\mu}_n$ converges weakly to a Borel probability measure $\mu$ on $[0,\infty)$.
Then the sequence $\{(\R^n, \|\cdot\|, \hat{\mu}_n)\}_{n=1}^\infty$ concentrates to the half line $([0, \infty), |\cdot|, \mu)$.
\end{rem}

Corollary \ref{cor:Hopf} leads to the concentration of the following concrete example.

\begin{cor}
Let $F =\R$, $\C$, or $\Hb$ and let $d:= \dim_\R{F}$.
Then the quotient space
\[
\frac{1}{\sqrt{n}} C^{dn}_{\beta}/U^F
\]
concentrates to the half line $([0,\infty), |\cdot|, \nu_\beta)$ as $n \to \infty$.
\end{cor}

We next prove Corollaries \ref{cor:normal} and \ref{cor:OD}.

\begin{proof}[Proof of Corollaries \ref{cor:normal}]
This follows from Theorem \ref{thm:main_cauchy} and \cite{MMG}*{Proposition 6.2(1)} directly.
\end{proof}

\begin{proof}[Proof of Corollary \ref{cor:OD}]
In this proof, we denote
\[
\diam(\nu_\beta; 1-\kappa) := \diam(([0,\infty), |\cdot|, \nu_\beta); 1-\kappa)
\]
for simplicity.
By the definition of the observable diameter, we see that
\[
\OD(([0,\infty), |\cdot|, \nu_\beta); -\kappa) = \diam(\nu_\beta; 1-\kappa).
\]
Theorem \ref{thm:main_cauchy} and the limit formula of the observable diameter in \cite{OS}*{Theorem 1.1} imply
\begin{align*}
\diam(\nu_\beta; 1-\kappa) &\le \liminf_{n\to\infty} \OD(\frac{1}{\sqrt{n}} C^n_\beta;-\kappa) \\
&\le \limsup_{n\to\infty} \OD(\frac{1}{\sqrt{n}} C^n_\beta;-\kappa) \le \lim_{\ep\to0+} \diam(\nu_\beta; 1-(\kappa-\ep))
\end{align*}
for any $\kappa > 0$. We prove
\begin{equation}\label{eq:recip_partial}
\lim_{\ep\to0+} \diam(\nu_\beta; 1-(\kappa-\ep)) = \diam(\nu_\beta; 1-\kappa)
\end{equation}
for any $\kappa>0$. Fix $\kappa > 0$ and take any $\delta > 0$ with $\diam(\nu_\beta;1-\kappa) < \delta$. There exists a Borel subset $A$ of $[0,\infty)$ such that
\[
\nu_\beta(A) \ge 1-\kappa \quad \text{ and } \quad \diam{A} < \delta.
\]
Let $\eta := \delta -\diam{A} > 0$. Then, since
\[
\nu_\beta(U_{\frac{\eta}{3}}(A)) > \nu_\beta(A) \ge 1-\kappa,
\]
we have
\[
\lim_{\ep\to0+} \diam(\nu_\beta; 1-(\kappa-\ep)) \le \diam{U_{\frac{\eta}{3}}(A)} \le \diam{A} + \frac{2}{3}\eta < \delta,
\]
which implies \eqref{eq:recip_partial}. We obtain this corollary.
\end{proof}

We now prove Theorem \ref{thm:main_cauchy_gene}.

\begin{proof}[Proof of Theorem \ref{thm:main_cauchy_gene}]
Let $\{r_n\}_{n=1}^\infty$ be a sequence of positive real numbers.
\begin{enumerate}
\item Assume that $r_n\sqrt{n} \to 0$. Take any $\ep > 0$. Lemma \ref{lem:rad_conv} implies
\[
\liminf_{n\to\infty} \nu^n_\beta(U_{\ep r_n^{-1}}(o)) = \liminf_{n\to\infty} (p_n)_\# \tilde{\nu}^n_\beta([0, \frac{\ep}{r_n\sqrt{n}})) \ge \nu_\beta([0, \infty)) = 1,
\]
where $o$ is the origin of $\R^n$.
Thus, for sufficiently large $n$, we have
\[
\nu^n_\beta(U_{\ep r_n^{-1}}(o)) \ge 1-\ep,
\]
which implies that $\prok(\nu^n_\beta, \delta_o) \le \ep$, where $\prok$ is defined with respect to $r_n \|\cdot\|$ on $\R^n$.
Therefore, by Proposition \ref{prop:mmg4.12}, we obtain
\[
\limsup_{n\to\infty}\square(r_nC^n_\beta, *) \le 2\ep.
\]
As $\ep \to 0$, the sequence $\{r_n C^n_\beta\}_{n=1}^\infty$ box-converges to the one-point mm-space $*$.

\item Assume that $r_n\sqrt{n} \to \lambda \in (0,\infty)$.
By Theorem \ref{thm:main_cauchy} and \cite{prod}*{Theorem 1.4}, the sequence $\{r_n C^n_\beta\}_{n=1}^\infty$ concentrates to the half line $([0, \infty), \lambda|\cdot|, \nu_\beta)$ scaled with $\lambda$. Let
\[
\nu_{\beta, \lambda}(dt) := \frac{2^{1-\frac{\beta}{2}}}{\Gamma(\frac{\beta}{2})} \frac{\lambda^\beta}{t^{\beta+1}} e^{-\frac{\lambda^2}{2t^2}} \, dt.
\]
Then $([0, \infty), \lambda|\cdot|, \nu_\beta)$ and $([0, \infty), |\cdot|, \nu_{\beta, \lambda})$ are mm-isomorphic to each other. Therefore we obtain the conclusion.

\item Assume that $r_n\sqrt{n} \to \infty$. By Corollary \ref{cor:OD}, we have
\[
\liminf_{n\to\infty} \OD(r_n C^n_\beta;-\kappa) = \liminf_{n\to\infty} (r_n\sqrt{n}) \OD(\frac{1}{\sqrt{n}} C^n_\beta;-\kappa) = \infty.
\]
Therefore $\{r_nC^n_\beta\}_{n=1}^\infty$ infinitely dissipates by Proposition \ref{prop:dissipate}.
\end{enumerate}
Therefore we obtain Theorem \ref{thm:main_cauchy_gene}.
\end{proof}

\begin{rem}
In \cite{ellipsoid}, the $n$-dimensional Gaussian space $\Gamma^n_{\{(n^{-1})^2\}} = n^{-1}\Gamma^n_{\{1^2\}}$ with variance $1/n^2$ concentrates to $*$ as $n \to \infty$ but it does not box-converge,
that is, there exist scaling orders of non-box-convergent L\'evy family.
On the other hand, for the Cauchy space, $\{r_nC^n_\beta\}_{n=1}^\infty$ is a L\'evy family if and only if it box-converges.
\end{rem}

\section{Proof of Theorem \ref{thm:non-box}}\label{sec:non-box}

In this section, we prove Theorem \ref{thm:non-box}. We need the following two lemmas.

\begin{lem}[\cite{MMG}*{Lemma 5.43}]\label{lem:partial}
If a sequence $\{X_n\}_{n=1}^\infty$ of mm-spaces box-converges to an mm-space $X$, then we have
\[
\diam(X;s) \leq \liminf_{n\to\infty} \diam(X_n;s) \leq \limsup_{n\to\infty}\diam(X_n;s) \leq \lim_{\delta\to 0+} \diam(X;s+\delta)
\]
for any $s >0$.
\end{lem}

\begin{lem}[\cite{ellipsoid}*{Lemma 4.5}] \label{lem:box-dom}
Let $\{X_n\}_{n=1}^\infty$ be a box-convergent sequence of mm-spaces and $\{Y_n\}_{n=1}^\infty$ a sequence of mm-spaces with $Y_n \prec X_n$.
Then, $\{Y_n\}_{n=1}^\infty$ has a box-convergent subsequence.
\end{lem}

\begin{proof}[Proof of Theorem \ref{thm:non-box}]
Let $R>0$ be the real number with
\[
\nu_\beta([0, R]) = \frac{1}{3}.
\]
We define maps $\phi_n \colon \R^n \to \R^n$ and $\phi\colon [0,\infty)\to [0,\infty)$ by
\[
\phi_n(x) := \begin{cases}
x & \text{if } \|x\| \le R, \\
\frac{R}{\|x\|} x & \text{if } \|x\| \ge R,
\end{cases}
\quad \text{and} \quad
\phi(r) := \begin{cases}
r & \text{if } r \le R, \\
R & \text{if } r \ge R.
\end{cases}
\]
Then $\phi_n$ and $\phi$ are 1-Lipschitz and $\|\phi_n(x)\| = \phi(\|x\|)$ for any $x \in \R^n$.

\begin{clm}
The sequence $\{(B^n_R(o), \|\cdot\|, (\phi_n)_\# \tilde{\nu}^n_\beta)\}_{n=1}^\infty$ concentrates to $([0,R], |\cdot|, \phi_\# \nu_\beta)$ as $n \to \infty$, where $B^n_R(o)$ is the closed ball in $\R^n$ centered at the origin $o$ with radius $R$.
\end{clm}

\begin{proof}
Let $p_n(x) := \|x\|$ for $x \in \R^n$. Then we have
\[
\prok((p_n)_\# (\phi_n)_\# \tilde{\nu}^n_\beta, \phi_\# \nu_\beta) = \prok(\phi_\# (p_n)_\# \tilde{\nu}^n_\beta, \phi_\# \nu_\beta) \le \prok((p_n)_\# \tilde{\nu}^n_\beta, \nu_\beta)
\]
since $\phi$ is 1-Lipschitz and \eqref{eq:lip_prok}.
Thus, by Theorem \ref{thm:main} and Lemma \ref{lem:rad_conv}, the sequence of
\[
K_{0,(p_n)_\# (\phi_n)_\# \tilde{\nu}^n_\beta}(S^{n-1}) = (B^n_R(o), \|\cdot\|, (\phi_n)_\# \tilde{\nu}^n_\beta)
\]
concentrates to $([0,R], |\cdot|, \phi_\# \nu_\beta)$ as $n \to \infty$.
\end{proof}

Now we suppose that $\{\frac{1}{\sqrt{n}}C^n_\beta\}_{n=1}^\infty$ has a box-convergent subsequence. Then, by Lemma \ref{lem:box-dom}, up to subsequence, $\{(B^n_R(o), \|\cdot\|, (\phi_n)_\# \tilde{\nu}^n_\beta)\}_{n=1}^\infty$ box-converges to  $([0,R], |\cdot|, \phi_\# \nu_\beta)$ as $n\to\infty$.
Since $\phi_\# \nu_\beta(\{R\}) = \frac{2}{3}$, we have
\[
\diam(\phi_\# \nu_\beta; \frac{2}{3}) = \diam\{R\} = 0.
\]
Moreover, if a Borel subset $A \subset B^n_R(o)$ satisfies $(\phi_n)_\# \tilde{\nu}^n_\beta(A) \ge \frac{1}{2}$, then
\[
A \cap \left\{x\in\R^n \midd \|x\| = R\right\} \neq \emptyset
\]
since $(\phi_n)_\# \tilde{\nu}^n_\beta(\{\|x\|<R\}) = \tilde{\nu}^n_\beta(\{\|x\|<R\}) \to \frac{1}{3}$ as $n\to\infty$.
In addition, if $\diam{A} < R$, then the set $A$ is in a half space of $\R^n$, which implies $(\phi_n)_\#\tilde{\nu}^n_\beta(A) < \frac{1}{2}$.
This is a contradiction.
Thus we have
\[
\liminf_{n\to\infty} \diam((\phi_n)_\# \tilde{\nu}^n_\beta; \frac{1}{2}) \ge R > 0.
\]
On the other hand, Lemma \ref{lem:partial} implies
\[
\limsup_{n\to\infty} \diam((\phi_n)_\# \tilde{\nu}^n_\beta; \frac{1}{2}) \le \lim_{s\to0+} \diam(\phi_\# \nu_\beta; \frac{1}{2}+s) \le  \diam(\phi_\# \nu_\beta; \frac{2}{3}) = 0.
\]
Thus $\{\frac{1}{\sqrt{n}}C^n_\beta\}_{n=1}^\infty$ must have no box-convergent subsequence. The proof is completed.
\end{proof}

\section{Several quantities for Cauchy space}\label{sec:inv}

In this section, we compute several geometric and analytic invariants for the Cauchy space and the limit space.
We have developed the general theory of invariants for mm-spaces in our previous paper \cite{EKM}.

\begin{dfn}[$N$-weighted Ricci curvature, cf.~\cite{Mil, Ohta}]
For an $n$-dimensional weighted Riemannian manifold $(M, g, \mu)$ with $\mu = \rho\vol_g$ for some positive smooth function $\rho$, and for a real number $N \in (-\infty, \infty)\setminus \{n\}$, the $N$-weighted Ricci curvature is given by
\begin{equation*}
\Ric_{\mu, N} = \Ric_{\rho, N} := \Ric_g + \Hess(-\log{\rho}) - \frac{\nabla \log{\rho}\otimes \nabla \log{\rho}}{N-n}.
\end{equation*}
\end{dfn}

\begin{prop}\label{prop:Ric}
We have
\begin{equation}\label{eq:Ric_n}
\Ric_{\tilde{\nu}^n_\beta, -\beta} = \frac{n(n+\beta)}{(1+n\|x\|^2)^2} ((1+n\|x\|^2)\Id - n x x^t)
\end{equation}
at any point $x \in \R^n$, where $\Id$ is the identity matrix and $xx^t=(x_ix_j)_{ij}$, under the natural identification of the Ricci tensor with a matrix,
and
\begin{equation}\label{eq:Ric_beta}
\Ric_{\nu_\beta, -\beta} = \frac{1}{t^4} + \frac{1}{(\beta+1)t^6}
\end{equation}
at any point $t \in [0,\infty)$.
\end{prop}

\begin{proof}
Let
\[
\rho_n(x) := \frac{n^{\frac{n}{2}}\Gamma(\frac{n+\beta}{2})}{\pi^{\frac{n}{2}}\Gamma(\frac{\beta}{2})}\frac{1}{(1+n\|x\|^2)^\frac{n+\beta}{2}}.
\]
We calculate
\[
\nabla(-\log{\rho_n}) = \frac{n+\beta}{2} \frac{2nx}{1+n\|x\|^2} = \frac{n(n+\beta)}{1+n\|x\|^2}x,
\]
and then
\begin{align*}
\Hess(-\log{\rho_n}) &= \frac{n(n+\beta)}{(1+n\|x\|^2)^2} ((1+n\|x\|^2)\Id - 2n x x^t), \\
\frac{\nabla(\log{\rho_n})\otimes \nabla(\log{\rho_n})}{-\beta-n} &= - \frac{n(n+\beta)}{(1+n\|x\|^2)^2} n x x^t.
\end{align*}
Therefore we obtain \eqref{eq:Ric_n}. We next prove \eqref{eq:Ric_beta}.
Let
\[
\rho(t) := \frac{2^{1-\frac{\beta}{2}}}{\Gamma(\frac{\beta}{2})}\frac{1}{t^{\beta+1}} e^{-\frac{1}{2t^2}}.
\]
We calculate
\[
\nabla(-\log{\rho}) = -\frac{1}{t^3} + \frac{\beta+1}{t},
\]
and then
\begin{align*}
\Hess(-\log{\rho}) &= \frac{3}{t^4} - \frac{\beta+1}{t^2}, \\
\frac{\nabla(\log{\rho})\otimes \nabla(\log{\rho})}{-\beta-1} &= - \frac{1}{\beta+1}\left(\frac{1}{t^6} - \frac{2(\beta+1)}{t^4} + \frac{(\beta+1)^2}{t^2}\right).
\end{align*}
Therefore we obtain \eqref{eq:Ric_beta}. The proof is completed.
\end{proof}

\begin{dfn}[Variance]
Let $X$ be an mm-space. The ({\it observable}) {\it variance} $V(X)$ of $X$ is defined by
\[
V(X) := \sup_{f \in \Lip_1(X)} \frac{1}{2} \int_X \int_X |f(x) - f(y)|^2 \, dm_X^{\otimes2}(x,y) \ (\le \infty).
\]
\end{dfn}

Note that $V \colon \X \to [0,\infty]$ is $2$-homogeneous, lower semicontinuous with respect to the box and concentration topologies, and monotone with respect to Lipschitz order (see \cite{EKM}*{Section 4.1}).

\begin{prop}\label{prop:var}
If $\beta > 2$, then we have
\begin{equation}\label{eq:var}
V(\frac{1}{\sqrt{n}} C^n_\beta) \le \frac{1}{\beta-2}
\end{equation}
for any $n$. Moreover, we have
\begin{equation}\label{eq:var_liminf}
\liminf_{n\to\infty} V(\frac{1}{\sqrt{n}} C^n_\beta) \ge V(([0, \infty), |\cdot|, \nu_\beta)) = \frac{1}{\beta-2} - \frac{\Gamma(\frac{\beta-1}{2})^2}{2\Gamma(\frac{\beta}{2})^2} > 0.
\end{equation}
\end{prop}

\begin{proof}
Since
\[
\int_{\R^n}\int_{\R^n} \langle x, y\rangle \, d(\nu^n_\beta)^{\otimes2}(x, y) = n \int_{\R}\int_{\R} xy \, d(\nu^1_\beta)^{\otimes2}(x, y) = n \left(\int_{\R} x \, d\nu^1(x)\right)^2 = 0,
\]
we have
\[
V(\frac{1}{\sqrt{n}}C^n_\beta) \le \frac{1}{2} \int_{\R^n}\int_{\R^n} \frac{1}{n} \|x-y\|^2 \, d(\nu^n_\beta)^{\otimes2}(x, y) =  \frac{1}{n} \int_{\R^n} \|x\|^2 \, d\nu^n_\beta(x) = \frac{1}{\beta - 2}.
\]
Indeed, we verify the last equality as follows.
\begin{align*}
\frac{1}{n} \int_{\R^n} \|x\|^2 \, d\nu^n_\beta(x) &= \frac{\Gamma(\frac{n+\beta}{2})}{n\pi^{\frac{n}{2}}\Gamma(\frac{\beta}{2})} \vol(S^{n-1})  \int_0^\infty \frac{r^{n+1}}{(1+r^2)^{\frac{n+\beta}{2}}} \, dr \\
&= \frac{\Gamma(\frac{n+\beta}{2})}{\pi^{\frac{n}{2}}\Gamma(\frac{\beta}{2})} \frac{2\pi^{\frac{n}{2}}}{\Gamma(\frac{n}{2})} \frac{\Gamma(\frac{n}{2}+1)\Gamma(\frac{\beta}{2}-1)}{2\Gamma(\frac{n+\beta}{2})} = \frac{1}{\beta-2}.
\end{align*}
Moreover, since the variance $V$ is lower semicontinuous with respect to the concentration topology, Theorem \ref{thm:main_cauchy} implies
\[
\liminf_{n\to\infty} V(\frac{1}{\sqrt{n}} C^n_\beta) \ge V(([0, \infty), |\cdot|, \nu_\beta)).
\]
and we have
\[
V(([0, \infty), |\cdot|, \nu_\beta)) = \int_0^\infty t^2 \, d\nu_\beta(t) - \left(\int_0^\infty t \, d\nu_\beta(t)\right)^2 = \frac{1}{\beta-2} - \frac{\Gamma(\frac{\beta-1}{2})^2}{2\Gamma(\frac{\beta}{2})^2}.
\]
Thus we obtain \eqref{eq:var_liminf}. The proof is completed.
\end{proof}

\begin{dfn}[Poincar\'e constant]
An mm-space $X$ satisfies the {\it $(2,2)$-Poincar\'e inequality} for constant $C > 0$ provided that for any bounded Lipschitz function $f$ on $X$,
\[
\frac{1}{2} \int_X \int_X |f(x) - f(y)|^2 \, dm_X^{\otimes2}(x,y) \le C^2 \int_X \lip_a(f)^2 \, dm_X,
\]
where $\lip_a(f)$ is the {\it asymptotic Lipschitz constant} defined by
\[
\lip_a(f)(x) := \lim_{r\to0} \sup_{y\neq y'\in U_r(x)} \frac{|f(y)-f(y')|}{d_X(y,y')}.
\]
The infimum of such a constant $C$ is denoted by $C_{2,2}(X)$ and is called the {\it $(2,2)$-Poincar\'e constant} of $X$.
\end{dfn}

Note that $C_{2,2} \colon \X \to [0,\infty]$ is $1$-homogeneous, lower semicontinuous with respect to the box and concentration topologies, and monotone with respect to the Lipschitz order (see \cite{EKM}*{Section 4.2}).

\begin{lem}
$C^n_\beta$ and $([0,\infty),|\cdot|, \nu_\beta)$ do not satisfy the $(2,2)$-Poincar\'e inequality, that is,
\[
C_{2,2}(C^n_\beta) = C_{2,2}(([0,\infty),|\cdot|, \nu_\beta)) = \infty.
\]
\end{lem}

\begin{proof}
By Muckenhoupt's criterion \cite{BGL}*{Theorem 4.5.1} and Proposition \ref{prop:Cauchy_dom}, it is sufficient to prove
\begin{align}
\sup_{x > 0} \nu^1_\beta([x,\infty)) \int_0^x \pi(1+t^2) \, dt &= \infty, \label{eq:Muck_Cauchy} \\
\sup_{x > m} \nu_\beta([x,\infty)) \frac{\Gamma(\frac{\beta}{2})}{2^{1-\frac{\beta}{2}}} \int_m^x t^{\beta+1} e^{\frac{1}{2t^2}} \, dt &= \infty, \label{eq:Muck_recip}
\end{align}
where $m$ is a median of $\nu_\beta$.

We first prove \eqref{eq:Muck_Cauchy}. For any $x>1$, we have
\begin{align*}
\nu^1_\beta([x,\infty)) \int_0^x \pi(1+t^2) \, dt
&\ge \int_x^\infty \frac{dt}{2t^2} \int_0^x (1+t^2) \, dt = \frac{1}{2} + \frac{1}{6}x^2,
\end{align*}
which implies \eqref{eq:Muck_Cauchy}. We next prove \eqref{eq:Muck_recip}.
We estimate
\[
\nu_\beta([x, \infty)) \ge \frac{2^{1-\frac{\beta}{2}}}{\Gamma(\frac{\beta}{2})} e^{-\frac{1}{2x^2}} \int_x^{\infty} \frac{dt}{t^{\beta+1}}
= \frac{2^{1-\frac{\beta}{2}}}{\Gamma(\frac{\beta}{2})} e^{-\frac{1}{2x^2}} \frac{1}{\beta x^{\beta}}
\]
and
\[
\frac{\Gamma(\frac{\beta}{2})}{2^{1-\frac{\beta}{2}}} \int_1^x t^{\beta+1} e^{\frac{1}{2t^2}} \, dt
\ge \frac{\Gamma(\frac{\beta}{2})}{2^{1-\frac{\beta}{2}}} \int_1^x t^{\beta +1} \, dt
= \frac{\Gamma(\frac{\beta}{2})}{2^{1-\frac{\beta}{2}}} \frac{1}{\beta+2}(x^{\beta+2}-1).
\]
Thus we have
\begin{equation*}
\nu_\beta([x, \infty)) \frac{\Gamma(\frac{\beta}{2})}{2^{1-\frac{\beta}{2}}} \int_1^x t^{\beta+1} e^{\frac{1}{2t^2}} \, dt
\geq \frac{1}{\beta(\beta+2)}e^{-\frac{1}{2x^2}} \left( x^2-\frac{1}{x^{\beta}} \right) \to \infty \text{ as } x\to\infty.
\end{equation*}
On the other hand, since $\nu_\beta$ is a Borel probability measure and $m$ is finite, we have
\begin{equation*}
\lim_{x\to\infty} \nu_\beta([x, \infty)) \frac{\Gamma(\frac{\beta}{2})}{2^{1-\frac{\beta}{2}}} \int_1^m t^{\beta+1} e^{\frac{1}{2t^2}} \, dt = 0.
\end{equation*}
These together imply that
\begin{equation*}
\lim_{x\to\infty} \nu_\beta([x, \infty))  \frac{\Gamma(\frac{\beta}{2})}{2^{1-\frac{\beta}{2}}} \int_m^x t^{\beta+1} e^{\frac{1}{2t^2}} \, dt =  \infty,
\end{equation*}
which implies \eqref{eq:Muck_recip}. The proof is completed.
\end{proof}

\section{Relation to the ergodic theory}\label{sec:ergodic}
In this section, we construct an infinite sequence of random variables that does not satisfy the law of large number while are identically distributed and have a finite moment, using the Cauchy distribution.
We start with recalling the Birkhoff ergodic theorem in this section.
\begin{thm}[\cite{W}*{Theorem 1.14}]\label{thm:BK_ergodic}
Let $(X, \Sigma, \mu)$ be a probability measure space and let $T \colon X \to X$ be a measure-preserving map, i.e., $T_\# \mu = \mu$. For any $f \in L^1(X, \mu)$, the time average
\[
\hat{f}(x) := \lim_{n\to\infty} \frac{1}{n} \sum_{i=0}^{n-1} f(T^i x)
\]
exists for almost all $x \in X$, where $T^i=T \circ T^{i-1}$ and $T^0 = \id_X$.
Moreover, $\hat{f}$ is given by the conditional expectation
\[
\E[f \mid \cI]
\]
almost everywhere on $X$, where $\cI := \left\{A \in \Sigma \midd T^{-1}(A) = A \right\}$ called the invariant $\sigma$-field of $T$.
\end{thm}

\begin{dfn}[Ergodicity]\label{dfn:ergodic}
Let $(X, \Sigma, \mu)$ be a probability measure space and let $T \colon X \to X$ be a measure-preserving map.
The map $T$ is said to be {\it ergodic} provided that for any $A \in \Sigma$ with $\mu(T^{-1}(A)\setminus A) = \mu(A \setminus T^{-1}(A)) = 0$, $\mu(A)$ is either $0$ or $1$.
\end{dfn}

\begin{thm}[\cite{W}*{Remark just below Theorem 1.14}]\label{thm:B_ergodic}
Under Theorem \ref{thm:BK_ergodic}, if $T$ is ergodic, then for any $f \in L^1(X, \mu)$, the time average $\hat{f}$ coincides with the mean $\E[f]$ almost everywhere on $X$.
\end{thm}

\begin{ex}\label{ex:SLLN}
Let $\mu$ be a Borel probability measure on $\R$. We consider the measure space $(\R^\N, \cB(\R^\N), \mu^{\otimes \infty})$.
The map $\Theta$ is defined by the shift
\[
\Theta(x_1, x_2, \ldots) = (x_2, x_3, \ldots), \quad (x_1, x_2, \ldots) \in \R^\N.
\]
The shift $\Theta$ is ergodic with respect to $\mu^{\otimes \infty}$.
Assume that
\[
\E[|x|] = \int_\R |x| \, d\mu(x) < \infty.
\]
Letting $f(x_1, x_2, \ldots) = x_1$, we have
\[
\lim_{n\to\infty} \frac{1}{n} \sum_{i=1}^n x_i = \int_\R x \, d\mu(x)
\]
for almost all $(x_1, x_2,\ldots) \in \R^\N$ by Theorem \ref{thm:B_ergodic}.
This is exactly the strong law of large numbers.
All coordinates $x_i$ of $\R^\N$ are independently and identically distributed by  $\mu$ with finite moment.
\end{ex}

As a corollary of Lemma \ref{lem:rad_conv} and Proposition \ref{prop:var}, we obtain the following theorem.

\begin{thm}\label{thm:non-ergodic}
Let $\nu^\infty_\beta$, $\beta>0$, be the Borel probability measure on $\R^\N$ characterized by
\[
(\pi_{n})_\# \nu^\infty_\beta = \nu^n_\beta,
\]
where $\pi_n\colon \R^\N \to \R^n$, $n\in\N$, are natural projections.
If $\beta > 2$, then the shift
\[
\Theta(x_1, x_2, \ldots) = (x_2,x_3,\ldots)
\]
is not ergodic with respect to $\nu^\infty_\beta$.
\end{thm}

\begin{proof}
Suppose that the shift $\Theta$ is ergodic with respect to $\nu^\infty_\beta$ for $\beta > 2$. We define a function $f \colon \R^\N \to \R$ by
\[
f(x_1, x_2, \ldots) = x_1^2.
\]
Note that $f$ is integrable with respect to $\nu^\infty_\beta$ by Proposition \ref{prop:var}. Then Theorem \ref{thm:B_ergodic} implies that
\[
\lim_{n\to \infty} \frac{1}{n} \sum_{i=1}^n x_i^2 = V(C^1_\beta) = \frac{1}{\beta-2}
\]
for almost all $(x_1, x_2, \ldots) \in \R^\N$ with respect to $\nu^\infty_\beta$. This means that
\[
(p_n)_\# \tilde{\nu}^n_\beta \to \delta_{\frac{1}{\sqrt{\beta-2}}} \quad \text{weakly}
\]
as $n \to \infty$, which contradicts Lemma \ref{lem:rad_conv}. The proof is completed.
\end{proof}

\begin{rem}
The existence and uniqueness of $\nu^\infty_\beta$ follows from the Kolmogorov extension theorem and \eqref{eq:consistency}.
All coordinates $x_i$ of $\R^\N$ are identically distributed by $\nu^1_\beta$ and the measure $\nu_\beta^1$ has a finite variance if $\beta>2$.
On the other hand, since $\nu^\infty_\beta$ is not a product measure, coordinates are not independently distributed to each other.
\end{rem}

\end{document}